\documentclass[11pt, a4paper]{article} 

\usepackage[english]{babel} 
\usepackage[utf8]{inputenc} 
\usepackage{vmargin} 
\setmargins{2.5cm}       
{1.5cm}                        
{16.5cm}                      
{23.42cm}                    
{10pt}                           
{1cm}                           
{0pt}                             
{2cm}                           

\usepackage{amsmath}
\usepackage{amssymb}
\usepackage{amsthm}

\usepackage[final]{graphicx}
\usepackage{soul}
\usepackage{vmargin}
\usepackage{braket}
\usepackage{enumerate}
\usepackage{enumitem} 
\setlist[enumerate,1]  {label={\rm (\roman*)}, leftmargin=1.5em}   
\setlist{noitemsep} 

\usepackage{empheq} 
\include{extsizes}
\usepackage{bigstrut}
\usepackage{float} 
\usepackage{mathrsfs} 
\usepackage{mdframed} 
\usepackage{xcolor} 

\usepackage[nottoc]{tocbibind}
\usepackage{amsthm}
\usepackage{wrapfig}
\usepackage{caption2}
\usepackage{esint}
\usepackage[title]{appendix} 
\usepackage{titlesec}
\usepackage{commath}
\usepackage{accents}

\catcode`@=11 \@addtoreset{equation}{section} \catcode`@=12

\newtheorem{theorem}{Theorem}[section]
\newtheorem{proposition}[theorem]{Proposition}
\newtheorem{lemma}[theorem]{Lemma}
\newtheorem{corollary}[theorem]{Corollary}

\newtheorem{remark}[theorem]{Remark}

\renewenvironment{proof}[1][Proof] {\noindent \textbf{#1.} }
{\  \rule{0.5em}{0.5em}\par \medskip}

\newcommand{\hole}{\mathscr{C}}
\newcommand{\myeq}[1]{\ensuremath{\stackrel{\text{#1}}{=}}}

\newcommand{\myleq}[1]{\ensuremath{\stackrel{\text{#1}}{\leq}}}
\newcommand{\mygeq}[1]{\ensuremath{\stackrel{\text{#1}}{\geq}}}

\newcommand{\defeq}{\vcentcolon=}
\newcommand{\eqdef}{=\vcentcolon}

\newcommand\restr[2]{\ensuremath{{#1}_{|_{#2}}}}

\renewcommand{\abs}[1]{\left| #1 \right|}
\newcommand{\D}{\displaystyle}
\newcommand{\eps}{\varepsilon}

\newcommand*\Lap{\mathop{}\!\mathbin\Delta}

\titleformat{\paragraph}
{\normalfont\normalsize\bfseries}{\theparagraph}{1em}{}
\titlespacing*{\paragraph}
{0pt}{3.25ex plus 1ex minus .2ex}{1.5ex plus .2ex}

\newcommand*\adh[1]{\overline{#1}} 

\newcommand{\RN}{{\mathbb{R}^N}}

\newcommand{\R}{\mathbb{R}}
\newcommand{\N}{\mathbb{N}}

\def\XXint#1#2#3{{\setbox0=\hbox{$#1{#2#3}{\int}$ }
		\vcenter{\hbox{$#2#3$ }}\kern-.580\wd0}}

 \usepackage[pagebackref]{hyperref}
\hypersetup{
  colorlinks=true,
  citecolor = blue, 
  linkcolor= blue,
}

\usepackage{mathtools}
\mathtoolsset{showonlyrefs=true} 

\begin{document}
	
	\title{Asymptotic behaviour of the heat equation in an exterior domain
		with general boundary conditions I. The case of
                integrable data.}

	\author{ Joaquín Domínguez-de-Tena${}^{*,1}$ \\
		An\'{\i}bal Rodr\'{\i}guez-Bernal\thanks{Partially supported
			by Projects PID2019-103860GB-I00 and  PID2022-137074NB-I00,  MICINN and   GR58/08
			Grupo 920894, UCM, Spain}\ ${}^{,2}$}

	\date{\today}
	\maketitle

	\setcounter{footnote}{2}
	\begin{center}
		Departamento de Análisis Matemático y  Matem\'atica Aplicada\\ Universidad
		Complutense de Madrid\\ 28040 Madrid, Spain \\ and \\
		Instituto de Ciencias Matem\'aticas \\
		CSIC-UAM-UC3M-UCM\footnote{Partially supported by ICMAT Severo Ochoa
			Grant CEX2019-000904-S funded by MCIN/AEI/ 10.13039/501100011033} , Spain 
	\end{center}
	
	\makeatletter
	\begin{center}
		${}^{1}${E-mail:
			joadomin@ucm.es}
		\\ 
		${}^{2}${E-mail:
			arober@ucm.es}
	\end{center}
	\makeatother

	\noindent {$\phantom{ai}$ {\bf Key words and phrases:}  Heat equation,
		exterior domain, mass loss, asymptotic profiles, asymptotic behaviour, decay rates, Dirichlet, Neumann,
		Robin boundary conditions.} 
	\newline{$\phantom{ai}$ {\bf Mathematical Subject Classification
			2020:} \
		35K05, 35B40, 35B30, 35E15}

	\begin{abstract}
		In this work, we study the asymptotic behaviour of solutions to the heat equation in exterior
		domains, i.e., domains which are the complement of a
                smooth  compact set in $\RN$. Different homogeneous
		boundary conditions are considered, including
                Dirichlet, Robin, and Neumann conditions for 
                integrable initial data in $L^1(\Omega)$. After taking
                into account the loss of mass of the solution through
                the boundary, depending on the boundary conditions, we describe the asymptotic spatial
                distribution of the remaining mass in terms of the
                Gaussian and of a suitable asymptotic profile
                function. We show that our results have optimal time
                rates. 
	\end{abstract}

	\normalsize

        
\section{Introduction}

  In this paper we consider the heat equation

  \begin{equation}
    \label{eq:intro_heat}
	\left\{
	\begin{aligned}
		u_t-\Lap u = 0 \quad & in \ \Omega\times(0,\infty) \\
		B(u)=0 \quad & on \ \partial\Omega\times (0,\infty) \\
		u=u_0 \quad & in \ \Omega\times\{0\} , 
	\end{aligned}	
	\right. 
\end{equation}
in   a connected unbounded exterior domain $\Omega$, that is, the
complement of a compact set $\hole$ that we denote the \emph{hole},
which is the closure of a bounded smooth set; hence,
$\Omega=\RN\backslash \hole$.  We will assume,  without loss of
generality, that  $0\in \mathring{\hole}$,
the interior of the hole, and observe that $\hole$ may have different
connected components, although $\Omega$ is connected. The boundary
conditions, to be made precise in Section \ref{sec:preliminares}, include
Dirichlet, Neumann and Robin one,  of the form $B(u)= \frac{\partial u}{\partial 
  n}+b u=0$ with $b >0$.

If we consider  nonnegative  initial data in  $L^1(\Omega)$, which
are  the most interesting  from  the physical and
probabilistic interpretation of the heat equation, 
then we see that the
hole and the boundary conditions imply that there is  some loss of
mass of the solution through the boundary of the hole in the case of
Dirichlet and Robin boundary conditions. Naturally, there is no loss
for Neumann one.  Actually,
integrating the equation in $\Omega$, we obtain 
\begin{equation} \label{eq:loss_of_mass_L1}
	\frac{d}{dt}\int_\Omega u(x,t)dx=\int_\Omega \Lap u(x,t)dx
        =\int_{\partial \Omega} \frac{\partial u}{\partial n}(x,t)dx
        . 
\end{equation}
So, since $u\geq 0$, for Dirichlet boundary conditions we have
$\restr{u}{\partial \Omega}=0$ and then 
$\frac{\partial u}{\partial n} \leq 0$ on $\partial \Omega$ and then
$\D \int_\Omega u(x,t)dx$ decreases in
time although we have no quantitative estimate of the decay. 
The same argument holds for Robin boundary conditions $\frac{\partial u}{\partial 
  n}+b  u=0$ with $b >0$, while for Neumann, $b=0$, the 
mass is conserved. This is in sharp contrast with the case
$\Omega=\R^{N}$ where the mass of every solution is conserved and is
due to the presence of the hole and the boundary conditions.

Therefore in  \cite{DdTRB23} the problem of 
understanding and determining the amount of mass lost for any given
solution was addressed. It turned out that the answer depends on  the
dimension. Indeed, the exact amount of mass lost for each initial data $u_{0} \in
L^{1}(\Omega)$ can be computed as follows: there exists a nonnegative function,
$\Phi$, denoted the \emph{asymptotic profile},
determined
by the domain and boundary conditions alone, such that the amount of
mass not lost through the hole by a solution with initial data $u_{0}\in
L^{1}(\Omega)$, that is, the \emph{asymptotic mass} of the solution, 
is given by
\begin{equation}
  \label{eq:asymptotic_mass}
  m_{u_0} \defeq \lim_{t\to\infty}\int_\Omega u(x,t)\, dx = \int_\Omega u_0(x) \Phi (x) \,
  dx . 
\end{equation}
Of course,  $\Phi \equiv 1$ for Neumann boundary conditions in any dimensions
(hence no loss of mass at all for any solution), while for Robin or
Dirichlet boundary conditions, if $N\leq 2$ then $\Phi \equiv 0$. That is,
all mass is lost through the boundary. On the other hand, if $N \geq
3$, then
\begin{equation}
	\label{eqn:estasymptoticprofile}
  1-\frac{C}{\abs{x}^{N-2}}\leq \Phi(x) \leq 1 \qquad 
  x\in \Omega 
\end{equation}
Hence, if  $N\geq 3$, then there is  a
certain remaining mass, while if $N\leq 2$, all the mass is 
lost through the hole. 
The function $\Phi $ is a harmonic function
in $\Omega$, $\Phi  \in C^2(\overline{\Omega})\cap
C^{\infty}(\Omega)$ and satisfies the boundary conditions 
$B(\Phi)\equiv 0$ on $\partial \Omega$. It can be constructed either
as the monotonically decreasing limit
\begin{equation}
  \label{eq:asymptotic_profile_parabolic}
	\Phi(x) =\lim_{t\to\infty} u(x,t; 1_\Omega) \quad x\in \Omega, 
\end{equation}
that is, the solution of \eqref{eq:intro_heat} with $u_{0} =
1_\Omega$, 
or as  the monotonically decreasing limit 
\begin{equation}
  \label{eq:asymptotic_profile_elliptic}
	\Phi (x)=\lim_{R\to\infty} \phi_R(x) \quad x\in \Omega, 
\end{equation}
where $ \phi_R$ are harmonic in $ \Omega_R\defeq \Omega\cap B(0,R)$
and satisfy $B(\phi_{R})(x) =0$ for $x\in\partial \Omega$ and $
\phi_{R}(x) =1$ if $\abs{x}=R$, 
see  \cite{DdTRB23} Section 3. Finally, rates of convergence in
\eqref{eq:asymptotic_mass} were also given in \cite{DdTRB23}.

Since it can be shown that the supremum of the solution tends to zero
as $t\to \infty$, see \eqref{eqn:LpLq_estimates_theta}, we see then
that, as $t\to \infty$,  the  asymptotic  mass of the initial data is diffused to
infinity, that is to the region $|x|\to \infty$. Hence, in  this paper  we address the question of the spatial distribution of
the asymptotic mass of the solution.

For the case of no hole, that is,  $\Omega=\R^{N}$ this
problem has been addressed in, e.g.,  \cite{duoandikoetxea}, or the
  first chapter in the book \cite{giga} or in   \cite{vazquez2017asymptotic}.  All these results exploit the
  explicit form of the heat kernel in $\R^{N}$ and the  outcome is that
  the mass of the solutions distributes in $\R^{N}$ as the total mass
  of the initial data times 
  the Gaussian heat kernel in $\R^{N}$. In \cite{duoandikoetxea},  the authors describe fine
  asymptotics of the solution in terms of finite  momentums of the initial
  datum, by showing that the more momentums the initial data has, the more
  asymptotic terms can be determined in terms of the momentums and
  derivatives of the heat kernel. 
For the case of an exterior domain, related results have been obtained
for Dirichlet boundary conditions for porous media equations in 
  \cite{quiros2007} for porous media or for non-local diffusion
  problems in  \cite{cortazar2012asymptotic}. For the heat equation we
  are only aware of the results in \cite{Herraiz1998}, again for
  Dirichlet boundary conditions, which we will
  discuss in Remark \ref{rem:Herraiz_result}.

As in \cite{duoandikoetxea, giga, vazquez2017asymptotic} we are going
to show that the Gaussian still describes the asymptotic spatial
distribution of mass with some corrections. First, as we have seen, we have a phenomenon of 
loss of mass. Therefore, the Gaussian will be multiplied by the 
asymptotic mass $m_{u_0}$ in \eqref{eq:asymptotic_mass}. Second, we  will also
have to take into account the boundary conditions on the hole and therefore
the asymptotic profile, $\Phi$ as in
\eqref{eq:asymptotic_profile_parabolic},
\eqref{eq:asymptotic_profile_elliptic},  will show up in the estimate as
well. To be more precise, our main result for integrable initial data,
that will be proved in Section \ref{sec:asympt-behav-Lpnorm}, see Theorem
\ref{thm:asympfinal},  is the following.

\begin{theorem}
	\label{thn:intro_asympfinal}

Assume  $N\geq 3$,   or  $N=2$ and we do not have Neumann boundary conditions,
and $u_0\in L^1(\Omega)$. Let $u(x,t)$ be
the solution of \eqref{eq:intro_heat}. Then, for any $1\leq p \leq \infty$, 
\begin{equation}
\lim_{t\to\infty} t^{\frac{N}{2}(1-\frac{1}{p})} \norm{u(\cdot,t) -
  m_{u_0} \Phi(\cdot) G(\cdot ,t)}_{L^p(\Omega)}=0,
\end{equation}
where $m_{u_0}=\int_{\Omega} \Phi(x)u_0(x)dx$ is the
asymptotic mass and $G(x,t) =\frac{e^{-\frac{\abs{x}^2}{4t}}}{(4\pi 
  t)^{\frac{N}{2}}}$ is the Gaussian. 
\end{theorem}
This will be obtained from interpolation from the extreme cases $p=1$
in Section \ref{sec:asympt-behav-L1norm} 
and $p=\infty$ in Section \ref{sec:asympt-behav-Linftynorm}, for which
we will also prove that the decay rate is 
optimal, see Theorems \ref{thm:optimal_rate_L1} and
\ref{thm:optimal_rate_Linfty}.

Besides these results, in Section \ref{sec:preliminares} we make
precise the boundary conditions in \eqref{eq:intro_heat} and introduce
some preliminary results on the solutions and on the asymptotic
profile $\Phi$. Also, we included some short  appendixes were we collected
some technical  auxiliar results needed for some of the proofs. 
In a forthcoming  paper \cite{DdTRB24},  we analyse the case of initial data in
$L^{p}(\Omega)$  with $1<p\leq\infty$. 

In the work in preparation  \cite{quiroscanizo24},   for Dirichlet boundary
    conditions, using entropy methods and assuming that $u_{0}$ has
    some finite momentums, similar results to Theorem
    \ref{thn:intro_asympfinal} are obtained together with finer convergence
    rates depending on the momentum of the initial data.

Throughout  this paper, we adopt
the convention of using $c$ and  $C$ to represent various constants 
which may change from line to line, and whose concrete value is not
relevant for the results.

\section{Notations and some  preliminary  results}
\label{sec:preliminares}

All along this paper we  consider an exterior domain  $\Omega =
\RN\backslash \hole$ as in the Introduction, that is, the
complement of a compact set $\hole$, the \emph{hole},
which is the closure of a bounded smooth set and we will assume 
$\partial \Omega$ is of class $C^{2,\alpha}$ for some $0<\alpha<1$. 
We will also assume,  without loss of generality,  $0\in \mathring{\hole}$,
the interior of the hole, and observe that $\hole$ may have different
connected components, although $\Omega$ is connected. 

In $\Omega$ we consider  the heat equation 
\begin{equation} \label{eq:heat_theta} 
	\left\{
	\begin{aligned}
		u_t-\Lap u = 0 \quad & in \ \Omega\times(0,\infty) \\
		B_\theta(u)=0 \quad & on \ \partial\Omega\times[0,\infty) \\
		u=u_0 \quad & in \ \Omega\times\{0\} , 
	\end{aligned}	
	\right. 
\end{equation}
where   we  consider Dirichlet, Robin or
Neumann homogeneous boundary conditions on $\partial \Omega$,  written in the form 
\begin{equation}
	\label{eqn:thetabc}
	B_\theta(u)\defeq \sin(\frac{\pi}{2}\theta(x))\frac{\partial u}{\partial n}+\cos(\frac{\pi}{2}\theta(x))u,
\end{equation}
where $\theta: \partial \Omega \longrightarrow [0,1]$ is of class
$C^{1,\alpha}(\partial \Omega)$ for some $0<\alpha<1$ and 
satisfies either  one of the
following cases in each connected component of $\partial \Omega$: 
\begin{enumerate}
	\item Dirichlet conditions: $\theta\equiv 0$
	\item Mixed Neumann and Robin conditions:
	$0<\theta\leq 1$. 
\end{enumerate}
In particular, if $\theta\equiv 1$ we recover Neumann
boundary conditions. 
In general, we will refer to these as homogeneous $\theta$-boundary
conditions.    Note that, by suitably choosing $\theta(x)$,
\eqref{eqn:thetabc} includes all boundary conditions of the 
form $\frac{\partial u}{\partial n}+b(x)u=0$. The restriction $0\leq
\theta \leq 1$ makes $b(x)\geq 0$ which is the standard dissipative
condition. The reason for these notations will be seen in the results
below about monotonicity of solutions with respect to $\theta$, see
\eqref{eqn:neugeqdir2} and \eqref{eq:comparison_kernels_theta}.

  In general we will use a superscript $\theta$ to denote anything
  related to \eqref{eq:heat_theta}.  For example, the semigroup of
  solutions to \eqref{eq:heat_theta} will be denoted by $S^\theta(t)$
  and the associated kernel by $k^\theta(x,y,t)$. Sometimes, we will add as subscript
  $\Omega$ to indicate the dependence of these objects in the domain.

Then, from the results in \cite{DdTRB23},  \eqref{eq:heat_theta} defines a semigroup of solutions as
$u(t;u_{0})= S^{\theta}(t) u_{0}$  for several classes of initial
data.  Actually the  semigroup $\{S^\theta(t)\}_{t>0}$   is  an order
preserving  semigroup of contractions in $L^p(\Omega)$ 
for $1\leq p\leq\infty$ which is $C^0$ if $p\neq \infty$ and
analytic if $1<p<\infty$. In particular, for  any
$u_0$ in those spaces, 
	\begin{displaymath}
		\abs{S^\theta(t)u_0(x)}\leq S^\theta(t)\abs{u_0}(x), \qquad x\in
		\Omega, \ t>0 . 
	\end{displaymath}

Also, 
for $1\leq p\leq \infty$, 
\begin{equation}
	\label{eqn:Sexchange}
	\int_\Omega fS^\theta(t)g = \int_\Omega gS^\theta(t)f  \qquad
	\mbox{for all} \  f\in L^p(\Omega), \   g\in L^q(\Omega) 
\end{equation}
where $q$ is the conjugate of $p$,  that is
$\frac{1}{p}+\frac{1}{q}=1$. Hence, for $1\leq p<\infty$, the semigroup in $L^{q}(\Omega)$
is the adjoint of the semigroup in $L^{p}(\Omega)$. In particular, the
semigroup in $L^{\infty}(\Omega)$ is weak-* continuous.

In addition, in  $L^p(\Omega)$ for  $1<p<\infty$, the generator of the
semigroup is  the Laplacian  with domain
\begin{equation}
	D^{p}(\Lap_\theta)=\{u\in W^{2,p}(\Omega) \ : \ B_\theta(u) =
        0  \ 	\mbox{on $\partial\Omega$} \} 
      \end{equation}
and is  a sectorial operator, see  \cite{denk2004new}  and
      \cite{DdTRB23} for a simple proof when $p=2$. Therefore, for $1\leq p<\infty$ the semigroup above provides the unique solution
      of \eqref{eq:heat_theta}, see e.g. Section 4.1 in \cite{pazy}.

If $p=\infty$, as in \cite{lunardi} Corollary 3.1.21 and
Corollary 3.1.24, the generator is  the Laplacian  with domain
\begin{equation} \label{eq:domain_Linfty}
	D^\infty(\Lap_\theta)=\{u\in \bigcap_{p\geq
          1}W^{2,p}_{loc}(\adh{\Omega}) \ : \ u, \Lap u \in
        L^\infty(\Omega), \ B_\theta(u) = 0 \ 	\mbox{on $\partial\Omega$} \}
      \end{equation}
      and is  also a sectorial operator with a non dense domain.
      
Note
that, by the Sobolev embeddings, $D^\infty(\Lap_\theta)\subset
C^{1+\alpha}(\adh{\Omega})$ for any $\alpha \in (0,1)$.

Moreover, the semigroup has an integral positive
kernel, that is, $k^\theta: 
\Omega \times \Omega \times (0,\infty) \to
(0,\infty)$ such that for all $1\leq p\leq\infty$ and
$u_{0}\in L^{p}(\Omega)$,  
\begin{equation}
	\label{eqn:ackrnpre}
	S^\theta(t)u_0(x)=\int_\Omega k^\theta(x,y,t)u_0(y)dy
	, \qquad x\in \Omega, \quad t>0. 
\end{equation}
In addition, $k^\theta(x,y,t)=k^\theta(y,x,t)$, a property that reflects
the selfadjointness of the semigroup.

If we consider  $S^{\theta_1}(t)$ and $S^{\theta_2}(t)$   the
semigroups above for different $\theta$-boundary
conditions  we have that if $	0 \leq \theta_1\leq
\theta_2\leq 1 $ then for $u_{0}\geq 0$ we have
\begin{equation}
	\label{eqn:neugeqdir2}
S^{\theta_1}(t)u_0\leq	S^{\theta_2}(t)u_0 \quad t>0 , 
\end{equation}
or equivalently,  the 
corresponding heat kernels satisfy 
\begin{equation}
\label{eq:comparison_kernels_theta}
  0< k^{\theta_1}(x,y,t)\leq k^{\theta_2}(x,y,t) \qquad
	x,y\in \Omega,  \ t>0. 
      \end{equation}
In particular,  for any $\theta$-boundary conditions we have  Gaussian upper
bounds for the  heat kernel of the form
	\begin{equation}
		\label{eqn:gyryabound}
		0 < k^\theta(x,y,t) \leq C\frac{e^{-\frac{\abs{x-y}^2}{4ct}}}{t^{N/2}} \qquad
		x,y\in \Omega , \quad t>0 
              \end{equation}
for some  constants $c,C>0$, since they hold for Neumann boundary
conditions (see \cite{gyryathesis} and  also \cite{gyrya2011neumann} Theorem 3.10), that is
for $\theta \equiv 1$, and \eqref{eq:comparison_kernels_theta}, 
see \cite{DdTRB23} Section 2.

The bounds above imply, by using Young's inequality for convolutions,

\begin{corollary}
  \label{cor:LpLq_estimates}
  For any $u_{0}\in L^{p}(\Omega)$ and $1\leq p\leq q \leq \infty$ we
  have 
          \begin{equation}
     \label{eqn:LpLq_estimates_theta}
     \norm{S^\theta(t)u_0}_{L^{q}(\Omega)}\leq
     \frac{C}{t^{\frac{N}{2}(\frac{1}{p}- \frac{1}{q})}} \norm{u_0}_{L^{p}(\Omega)} \quad t>0 . 
   \end{equation}
\end{corollary}

Concerning regularity of solutions and the kernels, we  can state the
following result. 

\begin{theorem}
	\label{thm:prop2}
The semigroup $S^\theta(t)$ has the following properties. 
	\begin{enumerate}
        \item
For $u_0\in L^p(\Omega)$, with $1\leq p\leq \infty$, 
    $u (x,t)= S^\theta(t)u_0(x)$ is a
    $C^{\infty}(\Omega\times(0,\infty))\cap C^{1}(\adh{\Omega}\times(0,\infty))$ solution of the heat
    equation, that is
    \begin{displaymath}
      \left\{
        \begin{aligned}
          u_t(x,t)-\Lap u(x,t) & = 0 &&
          \forall(x,t)\in\Omega\times(0,\infty)
          \\
          B_\theta(u)(x,t) & =0 && \forall x\in \partial \Omega, \
          \forall t>0 .
        \end{aligned}
      \right. 
    \end{displaymath}

  \item
    The integral kernel is analytic in time.
Furthermore, $k^\theta(\cdot,y, \cdot\cdot)$ belongs to
$C^{\infty}(\Omega\times(0,\infty))\cap C^{1}(\adh{\Omega}\times(0,\infty))$ and
satisfies the heat equation for any fixed $y$.

	\end{enumerate}
\end{theorem}
\begin{proof}
  For simplicity, we drop the superscript $\theta$ along this proof.\\
(i)  From Corollary \ref{cor:LpLq_estimates} we have that $S(t): \
  L^p(\Omega) \longrightarrow L^\infty(\Omega)$ is continuous for
  $t>0$. Since $S(t)=S(t/2)S(t/2)$, to study the
    regularity when $u_0\in L^p(\Omega)$ it is enough to study the
    regularity when $u_0\in L^\infty(\Omega)$.

Now, from the analyticity of the semigroup in $L^{\infty}(\Omega)$
and  the results of \cite{lunardi}, see \eqref{eq:domain_Linfty}, we have that $S(t):
L^\infty(\Omega) \longrightarrow D^\infty(-\Lap_\theta)$ is
continuous. In particular, by the Sobolev embeddings we have from
(\ref{eq:domain_Linfty})  that
$S(t): L^\infty(\Omega) \longrightarrow C^{1,\alpha}_{loc}(\adh{\Omega})$
for any $0<\alpha<1$. Then, using Lemma 3.1 of \cite{ACDRB04_linear}  we obtain
that, for any $u_0\in L^\infty(\Omega)$, $t\mapsto S(t)u_0$ is
analytic in $C^{1,\alpha}_{loc}(\adh{\Omega})$, so, in particular, $u: (x,t) \mapsto S(t)u_0(x)$
belongs to $C^{1}(\adh{\Omega}\times(0,\infty))$. 

In the same way,  since for all $k\in \N$, $t\mapsto \partial^{k}_{t}S(t)u_0=
-(-\Delta)^{k} S(t)u_{0}$ is analytic in
$C^{1,\alpha}_{loc}(\adh{\Omega})$ then we get  $u\in C^\infty(\Omega\times(0,\infty))$ using classical parabolic Schauder regularity (See for example \cite{librosouplet} Theorem 48.2).

\medskip 	
\noindent (ii)
Now, given $\varphi\in C^\infty_c(\Omega)$,  we use that
$S(t)\varphi=S(t/2)(S(t/2)\varphi)$. Then, for every $x\in \Omega$ and
$t>0$, 
	\begin{equation}
		\int_\Omega k(x,y,t)\varphi(y)dy=\int_\Omega k (x,z,t/2)\int_\Omega k(z,y,t/2)\varphi(y)dydz=\int_\Omega \left(\int_\Omega k (x,z,t/2)k(z,y,t/2) dz\right)\varphi(y)dy
	\end{equation}

As this is true for any $\varphi\in C^\infty_c(\Omega)$, we have
$k(x,y,t)=S(t/2)k(\cdot,y,t/2) (x)$, so from part (i), as $k(\cdot,y,t/2)$ is in $L^\infty(\Omega)$ we obtain $k(\cdot,y,\cdot \cdot)\in
C^{1}(\adh{\Omega}\times(0,\infty))\cap C^\infty(\Omega\times(0,\infty))$.
\end{proof}

Now we give estimates on the derivatives of the kernel and of the
solutions  that will be
needed further below. We start with estimates of time derivatives of
the kernel.

\begin{proposition}
Let $k^{\theta}(x,y,t)$ be the heat kernel in $\Omega$ with homogeneous
$\theta-$boundary conditions on $\partial \Omega$. Then 
	\begin{equation}
		\label{eqn:compnucleos}
		\abs{\frac{\partial^n}{\partial t^n} k^{\theta}(x,y,t)}\leq  C_1\frac{e^{-\frac{\abs{x-y}^2}{C_2t}}}{t^{\frac{N}{2}+n}} \qquad \forall x,y\in \Omega, \quad \forall t>0,
	\end{equation}
for some constants $C_1>0$ and $C_2>0$ which depend on $\Omega$, $\theta$ and $n$.
\end{proposition}
\begin{proof}
  From \eqref{eqn:gyryabound} we have
 \begin{equation}
      k^{\theta}(x,x,t/2)\leq a, \ \ k^{\theta}(y,y,t/2)\leq b, \ \ k^{\theta}(x,y,s)\leq a^{1/2}b^{1/2}c
    \end{equation}
    for all $s\in(t/2,3t/2)$, with  $a=b=\frac{C}{t^{N/2}}$, $c= e^{-\frac{\abs{x-y}^2}{Ct}}$.
Therefore, \cite{davies1997non}    Theorem 4,   implies 
	\begin{equation}
		\abs{\frac{\partial^n}{\partial t^n}k^{\theta}(x,y,t)}\leq  C\frac{a^{1/2}b^{1/2}c^{3/4}}{t^n}=C_1\frac{e^{-\frac{\abs{x-y}^2}{C_2t}}}{t^{N/2+n}}.
	\end{equation}
\end{proof}

Now we give estimates for solutions.

\begin{theorem}
\label{thm:est}

Let $u(x,t)=S^\theta(t)u_0(x)$ be the solution for the heat equation
with homogeneous $\theta-$boundary conditions, and initial datum
$u_0\in L^p(\Omega)$ with $1\leq p\leq \infty$. Then, if we denote
$d_x=d(x,\partial\Omega)$, we have 
	\begin{equation}
		\label{eqn:esteq1}
		\abs{\frac{\partial^{k+\abs{\beta}}}{\partial t^{k}\partial x^\beta}u(x,t)}\leq \frac{C_{\beta,k} \norm{u_0}_{L^p(\Omega)}}{t^{\frac{N}{2p}+k}\min (t^{1/2},d_x)^{\abs{\beta}}} \qquad \forall t>0, \ \ \forall x\in\Omega,
	\end{equation}
for any multi-index $\beta$ and non-negative integer $k$, where $C_{\beta,k}$ depends on $\beta$, $k$ and the domain $\Omega$.
	
\end{theorem}

\begin{proof}
It is enough to prove the result when $u_0\geq 0$. Otherwise
the positive and negative part of $u_0$ and apply the Theorem to each
part.  We proceed by   induction in $\abs{\beta}$.
	
For $|\beta|=0$,  recall that 
	\begin{equation}
		\frac{\partial^n}{\partial t^n}u(x,t) = \int_\Omega
                \frac{\partial^{n} k^{\theta}}{\partial
                  t^n}(x,y,t)u_0(y)dy 
	\end{equation}
        and, using \eqref{eqn:compnucleos},
	\begin{equation}
		\label{eqn:acotderu1}
		\abs{\frac{\partial^n}{\partial t^n}u(x,t)}\leq \int_\Omega \abs{\frac{\partial^n k^{\theta}}{\partial t^n}(x,y,t)}u_0(y)dy\leq \frac{C}{t^n}\int_\Omega \frac{e^{-\frac{\abs{x-y}^2}{4ct}}u_0(y)}{t^{N/2}}dy\leq C\frac{\norm{u_0}_{L^p(\Omega)}}{t^{\frac{N}{2p}+n}},
	\end{equation}	
where we have used Young's convolution inequality in the last
inequality. This proves the theorem with $\abs{\beta}=0$.

Now by induction in $\abs{\beta}$, assume 
\begin{equation}
\label{eqn:estindhyp}
\abs{\frac{\partial^{k+\abs{\beta}}}{\partial t^{k}\partial
    x^\beta}u(x,t)}\leq \frac{C_{\beta,k}
  \norm{u_0}_{L^p(\Omega)}}{t^{\frac{N}{2p}+k}\min
  (t^{1/2},d_x)^{\abs{\beta}}} \qquad \forall t>0, \ \ \forall
x\in\Omega, 
\end{equation}
for some multi-index $\beta$. Now, we use that
$v=\frac{\partial^{k+\abs{\beta}}}{\partial t^{k}\partial x^\beta}u$
is a solution of the heat equation in $\Omega$. Hence, given
$(x_0,t_0)\in \Omega\times \R_+$, we choose $R=\min
(t_0^{1/2},d_{x_0})$ and apply Schauder estimates in Theorem
\ref{thm:schauderest} with $Q=B(x_0,R/2)\times[t_0-R^2/2,t_0]$, so
that 
	\begin{equation}
		\label{eqn:thmesteqn1}
		\frac{R}{2}\abs{\partial_{x_i}v(x_0,t_0)}\leq C\norm{v}_{L^\infty(Q)}.
	\end{equation}
Using the induction hypothesis \eqref{eqn:estindhyp}, jointly with the
fact that if $(x,t)\in Q$ we have $t\geq t_0-R^2/2\geq t_0/2$ and
$d_x\geq d_{x_0}-d(x,x_0)\geq d_{x_0}-R/2\geq d_{x_0}/2$, we obtain 
	\begin{equation}
		\label{eqn:thmesteqn2}
		\norm{v}_{L^\infty(Q)}\leq \frac{C
                  \norm{u_0}_{L^p(\Omega)}}{t_0^{\frac{N}{2p}+k} \min (t_0^{1/2},d_{x_0})^{\abs{\beta}}}.
	\end{equation}
	Therefore, combining \eqref{eqn:thmesteqn1} and \eqref{eqn:thmesteqn2} we obtain
\begin{equation}
\label{eqn:estindhyp2}
\abs{\frac{\partial}{\partial
    x_i}\frac{\partial^{k+\abs{\beta}}}{\partial t^{k}\partial
    x^\beta} u(x_0,t_0)}\leq \frac{C
  \norm{u_0}_{L^p(\Omega)}}{t_0^{\frac{N}{2p}+k} \min  (t_0^{1/2},d_{x_0})^{\abs{\beta}+1}}, 
\end{equation}
which is the desired result for an additional spatial
derivative. Hence, by induction we obtain the general result. 

\end{proof}

\begin{remark}
 The previous estimates are not always sharp. A good
  discussion of the decay rates of the derivatives of the heat kernel
  can be found in \cite{ishige2007} and the references therein.
\end{remark}

Finally we present some  estimates on the  decay rates for
the derivatives of the asymptotic profile, $\Phi^\theta$, defined as
in \eqref{eq:asymptotic_profile_parabolic},
\eqref{eq:asymptotic_profile_elliptic}. Recall that if $\theta\equiv 1$ then
$\Phi^{1}\equiv 1$ and if $N=2$ and $\theta\not\equiv 1$, that is except
for Neumann boundary conditions, then $\Phi^{\theta}\equiv 0$. 

\begin{proposition}
	\label{prop:estpro}

Let $N\geq 3$  and $\Phi^\theta$ its asymptotic profile for $\theta-$boundary conditions. Then, there exists $C>0$ such that
\begin{equation}
		\label{eqn:boundphi}
		1-\frac{C}{\abs{x}^{N-2}}\leq \Phi^{\theta}(x) \leq 1 \qquad \forall x\in \Omega.
	\end{equation}	
	In addition, for any multi-index $\abs{\beta}\neq 0$, if $\Phi^\theta \in C^{\abs{\beta}}(\adh{\Omega})$ (which is true if $\partial\Omega$ and $\theta$ are sufficiently regular), there exists $C_\beta>0$ such that
	\begin{equation}
		\label{eqn:boundpsi}
		\abs{D^\beta\Phi^\theta(x)}\leq \frac{C_{\beta}}{\abs{x}^{N-2+\abs{\beta}}} \qquad x\in \Omega.
	\end{equation}
\end{proposition}

\begin{proof}
  The proof of \eqref{eqn:boundphi} can be found in
  \cite{DdTRB23} Proposition 4.8. To prove
  \eqref{eqn:boundpsi} for $\beta \neq 0$, we will prove
  \eqref{eqn:boundpsi} for $\psi=1-\Phi^\theta$. By hypothesis we have
  that $\psi$ is $C^{\abs{\beta}}(\adh{\Omega})$. Therefore, we only
  have to prove  \eqref{eqn:boundpsi} for $\abs{x}$ large and
  for this 
  we will use induction in $\beta$. For $\beta=0$, from \eqref{eqn:boundphi}, we
  already  have
  $\abs{\psi}\leq \frac{C}{\abs{x}^{N-2}}$. Now, assume the result is
  true for any multi-index $\beta$. We have that
  $v\defeq D^{\beta} \psi$ is
  a harmonic function in $\Omega$. Let $R>0$ such that
  $\partial \Omega\subset B(0,R)$. Given $\abs{x}\geq 2R$, with
  $B_1=B(x,\abs{x}/4)$ and $B_2=B(x,\abs{x}/2)\subset \Omega$ we use
  the classical Schauder estimates of Theorem \ref{thm:4.6} to obtain
	\begin{equation}
		\frac{\abs{x}}{4}\abs{\partial_{x_i}v(x)}\leq C \norm{v}_{L^\infty(B_2)}\leq \frac{C}{\left(\frac{\abs{x}}{2}\right)^{N-2+\abs{\beta}}},
	\end{equation}
	where $C>0$ is independent of $x$. Then,
	\begin{equation}
		\label{eqn:propesteq2}
		\abs{\frac{\partial}{\partial x_i} D^{\beta} \psi (x)}\leq \frac{C}{\abs{x}^{N-2+\abs{\beta}+1}}.
	\end{equation}
	As this is true for any $x_i$, we have proved the induction hypothesis.
	
	The fact that $\Phi^\theta\in C^{\abs{\beta}}(\adh{\Omega})$ if $\partial \Omega$ and $\theta$ are sufficiently regular, derives from higher order regularity estimates up to the boundary for harmonic functions (see for example \cite{Mikhailov} Section IV.2).
\end{proof}


\section{Asymptotic behavior in the $L^{1}(\Omega)$ norm}
\label{sec:asympt-behav-L1norm}

As mentioned before, integrable data  is the most interesting case from  the physical and
probabilistic interpretation of the heat equation. As we already know the exact amount of mass lost
through the hole, see \eqref{eq:asymptotic_mass},  our goal is to describe how the remaining mass
distributes spatially  in $\Omega$ as $t\to \infty$.
As mentioned in the Introduction, for the case of $\Omega=\R^{N}$ this
problem has been addressed in, e.g.,  \cite{duoandikoetxea}, or the
  first chapter in the book \cite{giga} or in   \cite{vazquez2017asymptotic}.  All these results exploit the
  explicit form of the heat kernel in $\R^{N}$ and the  outcome is that
  the mass of the solutions distributes in $\R^{N}$ as a multiple of
  the Gaussian heat kernel in $\R^{N}$. In \cite{duoandikoetxea},  the authors describe fine
  asymptotics of the solution in terms of the momentums of the initial
  datum, by showing that the more momentums the initial data has, the more
  asymptotic terms can be determined in terms of the momentums and
  derivatives of the heat kernel.

Without any  assumptions on the momentums of the initial data, the result
in references  \cite{duoandikoetxea, giga, vazquez2017asymptotic} for the problem in
$\RN$ is the following one, that states that  the mass of the solution, $M$, distributes in space, in a
     first order approximation, as  by $M$ times the Gaussian.

\begin{theorem}
	\label{thm:jlvazquez}

	Let $u_0\in L^1(\mathbb{R}^N)$ and
        $M=\int_{\mathbb{R}^N}u_0(x)dx$. Then, if we denote
        $u(t)=S_{\mathbb{R}^N}(t)u_0$, the solution of the heat
        equation in $\R^{N}$ with initial data $u_{0}$, we have 
	\begin{equation}
		\label{eqn:jlvazquez1}
		\lim_{t\to\infty}\norm{u(t)-M
                  G(\cdot,t)}_{L^1(\mathbb{R}^N)} =  0  
	\end{equation}
where $G(x,t) =\frac{e^{-\frac{\abs{x}^2}{4t}}}{(4\pi
  t)^{\frac{N}{2}}}$ is the Gaussian.

     Furthermore,
\begin{equation}
	\label{eqn:jlvazquez2}
	\lim_{t\to\infty}t^{N/2}\norm{u(t)-M
           G(\cdot,t) }_{L^\infty(\mathbb{R}^N)} = 0  
\end{equation}
and for every $1< p <\infty$,
\begin{equation}
	\label{eqn:jlvazquez3}
	\lim_{t\to\infty}t^{\frac{N}{2}(1-\frac{1}{p})}\norm{u(t)-M
          G(\cdot,t) }_{L^p(\mathbb{R}^N)} = 0. 
\end{equation}
\end{theorem}

We are going to prove a similar  result for an exterior
domain with homogeneous $\theta$-boundary conditions. However, as we have seen, we have a phenomenon of 
loss of mass. Therefore, we will have to substitute $M$ by the
asymptotic mass $m_{u_0}$ in \eqref{eq:asymptotic_mass}. We will also
have into account the boundary conditions on the hole and therefore
the asymptotic profile, $\Phi^{\theta}$ as in
\eqref{eq:asymptotic_profile_parabolic},
\eqref{eq:asymptotic_profile_elliptic},  will show up in the estimate as
well. 

In this  section we will prove a result
in  the $L^1(\Omega)$  norm as in \eqref{eqn:jlvazquez1}. In Section
\ref{sec:asympt-behav-Linftynorm} we will proof a result in the
$L^{\infty}(\Omega)$  norm as in \eqref{eqn:jlvazquez2} and then in
Section \ref{sec:asympt-behav-Lpnorm}, by
interpolation, as in \eqref{eqn:jlvazquez3}, see Theorem \ref{thm:asympfinal} below.

First of all, we will prove the following lemma which states that the
$\theta$-heat kernel in $\Omega$ and the kernel in $\RN$ are similar in
$L^1(\Omega)$ when the source point $y$ is far away from the hole. The
Lemma is stated for $N\geq 3$. Notice that
\eqref{eqn:lemma:kerncoml1biseq3} below is also true for $N\leq 2$ but gives
no interesting  information because the asymptotic profile is $\Phi^0\equiv
0$. Also, recall that the heat kernel in $\R^{N}$ is given by $k_{\mathbb{R}^N}(x,y,t)= G(x-y,t)$. 

\begin{lemma}
	\label{lemma:kerncoml1bis}
Assume  $N\geq 3$ and let  $k^{\theta}(x,y,t)$ be the heat kernel for $\theta-$boundary conditions
and $k_{\mathbb{R}^N}(x,y,t)$ the heat kernel in the whole space. Then 
\begin{equation}
\label{eqn:lemma:kerncoml1biseq3}
\int_\Omega \abs{k^{\theta}(x,y,t)-k_{\mathbb{R}^N}(x,y,t)}dx\leq
2(1-\Phi^0(y))+\int_\hole k_\RN(x,y,t)dx \qquad y\in \Omega, 
\end{equation}
where $\Phi^0$ is the asymptotic profile of $\Omega$ for Dirichlet
boundary conditions. In particular 
	\begin{equation}
		\label{eqn:lemma:kerncoml1biseq1}
		\limsup_{t\to\infty}\int_\Omega
                \abs{k^{\theta}(x,y,t)-k_{\mathbb{R}^N}(x,y,t)}dx\leq
                2(1-\Phi^0(y)) \qquad y\in \Omega . 
	\end{equation}
	Furthermore, for all $x\in\Omega$, 
	\begin{equation}
		\label{eqn:lemma:kerncoml1biseq2}
		\limsup_{t\to\infty}\limsup_{\abs{y}\to
                  \infty}\int_\Omega
                \abs{k^{\theta}(x,y,t)-k_{\mathbb{R}^N}(x,y,t)}dx=
           \limsup_{\abs{y}\to
                  \infty}     \limsup_{t\to\infty}\int_\Omega
                \abs{k^{\theta}(x,y,t)-k_{\mathbb{R}^N}(x,y,t)}dx =0
                . 
	\end{equation}
\end{lemma}
\begin{proof}
If we denote $k^{0}(x,y,t)$ the heat kernel with Dirichlet boundary
conditions, we have, using \eqref{eq:comparison_kernels_theta}, that 
	\begin{equation}
		\label{eqn:lemmaKNKReq1bis}
		k^{\theta}(x,y,t)\geq k^{0}(x,y,t) \qquad
                x,y\in\Omega,  \quad  t>0.
	\end{equation}
	Then, if for fixed $y\in\Omega$, we denote $\Omega^-(t)\defeq
        \{x\in\Omega : k^{\theta}(x,y,t) \leq
        k_{\mathbb{R}^N}(x,y,t)\}$ and
        $\Omega^+(t)=\Omega\backslash\Omega^-(t)$, we have 
	$$
	\int_{\Omega^-(t)}
        (k_{\mathbb{R}^N}(x,y,t)-k^{\theta}(x,y,t))dx
        \myleq{\eqref{eqn:lemmaKNKReq1bis}}  \int_{\Omega^-(t)} (k_{\mathbb{R}^N}(x,y,t)-k^{0}(x,y,t))dx
	$$	
	\begin{equation}
		\label{eqn:lemmaKNKReq2bis}	
		\myleq{\eqref{eq:comparison_kernels_Dirichlet_RN}}
                \int_{\Omega} (k_{\mathbb{R}^N}(x,y,t)-k^{0}(x,y,t))dx
                \leq 1-\int_{\Omega} k^{0}(x,y,t)dx . 
	\end{equation}
From the symmetry of the kernel and
\eqref{eq:asymptotic_profile_parabolic} we have that 
$\int_\Omega k^0(x,y,t)dx=\int_\Omega k^0(y,x,t)dx=S^{0}(t)1_\Omega(y)$ decays
monotonically  in $t$ to $\Phi^0(y)$.

Then, we have that
\begin{equation}
	\label{eqn:lemmaKNKReq5bis}
	\int_{\Omega^-(t)} (k_{\mathbb{R}^N}(x,y,t)-k^{\theta}(x,y,t))dx \leq 1-\Phi^0(y), \qquad \forall t>0, \ \forall y\in \Omega.
\end{equation}
	
In addition, as $\int_{\Omega}k^\theta(x,y,t)dx = S^{\theta}(t)1_\Omega\leq 1$,
\begin{displaymath}
\int_\Omega k^{\theta}(x,y,t)dx \leq 1= \int_\RN
k_\RN(x,y,t)dx=\int_\Omega k_\RN(x,y,t)dx+\int_{\hole} k_\RN(x,y,t)dx, 
\end{displaymath}
then 
	\begin{equation}
		\label{eqn:lemmaKNKReq3bis}
		\int_{\Omega} (k^{\theta}(x,y,t)-k_{\mathbb{R}^N}(x,y,t))dx \leq \int_{\hole}k_{\mathbb{R}^N}(x,y,t)dx\eqdef A(y,t).
	\end{equation}
	Hence, \eqref{eqn:lemmaKNKReq3bis} implies that
	\begin{equation}
		\label{eqn:lemmaKNKReq4bis}
		\begin{aligned}
			\int_{\Omega^+(t)} (k^{\theta}(x,y,t)-k_{\mathbb{R}^N}(x,y,t))dx & =\int_{\Omega} (k^{\theta}(x,y,t)-k_{\mathbb{R}^N}(x,y,t))dx-\int_{\Omega^-(t)} (k^{\theta}(x,y,t)-k_{\mathbb{R}^N}(x,y,t))dx \\
			& \myleq{\eqref{eqn:lemmaKNKReq3bis}}		
			 A(y,t)+\int_{\Omega^-(t)} (k_{\mathbb{R}^N}(x,y,t)-k^{\theta}(x,y,t))dx.
		\end{aligned}
	\end{equation}
	Then,
	$$
	\int_{\Omega} \abs{k^{\theta}(x,y,t)-k_{\mathbb{R}^N}(x,y,t)}dx 
	$$
	$$
	=\int_{\Omega^+(t)} (k^{\theta}(x,y,t)-k_{\mathbb{R}^N}(x,y,t))dx+\int_{\Omega^-(t)} (k_{\mathbb{R}^N}(x,y,t)-k^{\theta}(x,y,t))dx
	$$
	\begin{equation}
		\myleq{\eqref{eqn:lemmaKNKReq4bis}}
                2\int_{\Omega^-(t)}
                (k^{\theta}(x,y,t)-k_{\mathbb{R}^N}(x,y,t))dx+A(y,t)\myleq{\eqref{eqn:lemmaKNKReq5bis}}
                2(1-\Phi^0(y))+A(y,t), 
	\end{equation}
        which is \eqref{eqn:lemma:kerncoml1biseq3}.  Therefore,
        \eqref{eqn:lemma:kerncoml1biseq1} 
follows because $\int_{\hole} k_\RN(x,y,t)dx\to 0$ when $t\to
\infty$. 

To obtain \eqref{eqn:lemma:kerncoml1biseq2} we  have $\abs{\int_{\hole} k_\RN(x,y,t)dx}\leq
\abs{\hole}\norm{k_\RN(\cdot,y,t)}_{L^\infty(\RN)} \leq \abs{\hole}(4\pi t)^{-N/2}$, so we
obtain that $A(y,t)$ decays in $t$ uniformly in
$y\in\Omega$. Furthermore, as
$\lim_{\abs{y}\to\infty}\Phi^0(y)=1$ (See Proposition
\ref{prop:estpro}), we obtain \eqref{eqn:lemma:kerncoml1biseq2} from
\eqref{eqn:lemma:kerncoml1biseq3}. 
\end{proof}

As a consequence of the  previous lemma, if the initial datum is
supported far away from the hole, its asymptotic behaviour is similar
to the Gaussian, as the  following lemma shows.  

\begin{lemma}
	\label{lemma:lemmaprev2bis}
Let $N\geq 3$  and
$\varepsilon>0$. Then, there exists an $R>0$ such that, for any
$u_0\in L^1(\Omega)$ with
$supp(u_0)\subset
\mathbb{R}^N\backslash B(0,R)\subset \Omega$ and $M=\int_\Omega u_0$, there exists a $T\geq 0$
such that 
	\begin{equation}
		\label{lemma:pre1eq7bis}
		\norm{S^\theta(t)u_0- M G(\cdot ,t)}_{L^1(\Omega)}\leq
                \varepsilon \norm{u_0}_{L^1(\Omega)} \qquad \forall
                t>T. 
	\end{equation}	 
\end{lemma}
\begin{proof}
	As $N\geq 3$, we have that $\Phi^\theta(x)\to 1$ as $\abs{x}\to \infty$. Therefore, given $\varepsilon>0$, we can use Lemma \ref{lemma:kerncoml1bis} to prove that there exists a $T_0>0$ and an $R>0$ such that 
	\begin{equation}
		\label{lemma:pre1eq1bis}
		\int_\Omega \abs{k_{\mathbb{R}^N}(x,y,t)-k^\theta(x,y,t)}dx\leq \frac{\varepsilon}{2} \qquad t\geq T_0,
	\end{equation}
	for every $\abs{y}\geq R$.
        
Now, let $u_0\in L^1(\Omega)$ such that $supp(u_0)\subset \RN\backslash
B(0,R)$ and let   $\int_\Omega u_0=M$. Then, for  $t\geq T_0$, 
	\begin{equation}
				\label{lemma:pre1eq3bis}
		\begin{aligned}
			&\int_\Omega\abs{S_{\mathbb{R}^N}(t)u_0(x)-S^{\theta}(t)u_0(x)}dx
                        \leq  \int_\Omega\int_{\Omega}
                        \abs{k_{\mathbb{R}^N}(x,y,t)-k^\theta(x,y,t)} \abs{u_0(y)}dydx \\
			&=\int_{\RN\backslash B(0,R)}\int_\Omega \abs{k_{\mathbb{R}^N}(x,y,t)-k^\theta(x,y,t)}\abs{u_0(y)}dxdy \myleq{\eqref{lemma:pre1eq1bis}} \frac{\varepsilon}{2}\int_{\RN\backslash B(0,R)} \abs{u_0(y)}dy \leq \frac{\varepsilon}{2}\norm{u_0}_{L^1(\Omega)}.
		\end{aligned}	
	\end{equation}
	Now, using Theorem \ref{thm:jlvazquez}, we have that there
        exists a $T\geq T_0$ such that, extending $u_{0}$ by zero
        outside $\Omega$, 
	\begin{equation}
		\label{lemma:pre1eq6bis}
		\int_{\Omega}\abs{S_{\mathbb{R}^N}(t)u_0(x)-M
                  G(x ,t)}dx\leq
                \int_{\mathbb{R}^N}\abs{S_{\mathbb{R}^N}(t)u_0(x)-M G(x,t)}dx\myleq{Thm \ref{thm:jlvazquez}} \frac{\varepsilon}{2}\norm{u_0}_{L^1(\Omega)} \qquad \forall t\geq T.
	\end{equation}
	Finally, combining \eqref{lemma:pre1eq6bis} and \eqref{lemma:pre1eq3bis} we obtain \eqref{lemma:pre1eq7bis}
\end{proof}

Let us now state the asymptotic result for a general $u_0$ which is
analogous to \eqref{eqn:jlvazquez1}.

\begin{theorem}
\label{thm:asymL1Rbis}

Let  $u_0\in L^1(\Omega)$ and $m_{u_0}\defeq
\int_\Omega \Phi^\theta(x)u_0(x)dx$ the asymptotic mass of the
solution of the   heat equation \eqref{eq:heat_theta} with
$\theta$-boundary conditions. Then if $N\geq 3$ or if $N=2$ and
$\theta\not\equiv 1$, that is, except for Neumann boundary conditions,
we  have that 
	\begin{equation}
		\label{eqn:thm:asymL1Rbiseq1}
		\lim_{t\to\infty}\norm{S^\theta(t)u_0-m_{u_0} G(\cdot,t)}_{L^1(\Omega)}=0.
	\end{equation}
Hence, the mass of the solution decays asymptotically to $m_{u_0}$ and
distributes in space, in a first order approximation, as described by
$m_{u_0}$ times the Gaussian. 
\end{theorem}
\begin{proof}
If $N=2$ and $\theta\not\equiv 1$ from  \cite{DdTRB23} Theorem
  4.9, we have $\Phi^{\theta} =0$ and then $m_{u_{0}}=0$. From
  \eqref{eq:asymptotic_mass} this implies that if $u_{0}\geq 0$, then  the solution converges to
  $0$ in $L^{1}(\Omega)$. By splitting $u_{0}$ in the positve and
  negative parts, we get $\lim_{t \to \infty} S^{\theta}(t) u_{0} = 0$
  in  $L^{1}(\Omega)$, which proves \eqref{eqn:thm:asymL1Rbiseq1}.

  For $N\geq 3$ the idea of the proof is the following:
we will let the time pass so that the solution has lost most of its
mass through the hole  and the remaining mass is far away from the hole so we can
neglect the mass close to it and use Lemma \ref{lemma:lemmaprev2bis}
with the rest. The fact that we will let time pass,  will make the
solution lose its mass, so the asymptotic mass will appear. 
	
	Let $\varepsilon>0$. Firstly, we take the $R>0$ from Lemma \ref{lemma:lemmaprev2bis}.
Now, we denote $u(x,t)=S^\theta(t)u_0$. Due to
\eqref{eqn:LpLq_estimates_theta}, we know that $u(t)$ decays in
$L^\infty(\Omega)$. Therefore, there exists a $T_0\geq 0$ such that 
	\begin{equation}
		\label{eqn:thmasyeq0bis}
		\int_{B(0,R)\cap \Omega} \abs{u(x,t)}dx \leq \varepsilon \qquad \forall t\geq T_0.
	\end{equation}	
	Secondly, we take $m_{u_0}$ the asymptotic of $u_0$ as  in \eqref{eq:asymptotic_mass}. Then, there exists a $T_1\geq T_0$ such that
	\begin{equation}
		\label{eqn:thmasyeq1bis}
		\abs{\int_\Omega u(t)-m_{u_0}}\leq  \varepsilon \quad \forall t\geq T_1.
	\end{equation}
Now we define $v(t)\defeq S^\theta(t) (\chi_Ru(T_1))$ where $\chi_R$
is the characteristic function of $\RN\backslash B(0,R)$. Due to
\eqref{eqn:thmasyeq0bis} we have that 
	\begin{equation}
		\label{eqn:thmasyeq-2bis}
		\norm{v(0)-u(T_1)}_{L^1(\Omega)}=\norm{u(T_1)}_{L^1(\Omega\cap
                  B(0,R))}\leq \varepsilon . 
	\end{equation}
 Therefore, as $S^\theta(t)$ is a contraction semigroup in $L^1(\Omega)$,
	\begin{equation}
		\label{eqn:thmasyeq-1bis}
		\norm{v(t)-u(t+T_1)}_{L^1(\Omega)}\leq \varepsilon \qquad \forall t\geq 0.
	\end{equation}
In addition, if we combine \eqref{eqn:thmasyeq-2bis} with \eqref{eqn:thmasyeq1bis} we have
	\begin{equation}
		\label{eqn:thmasyeq1bisbis}
	\abs{\int_\Omega v(0)-m_{u_0}}\leq 2\varepsilon.
\end{equation}
	As $v(0)$ is supported outside $B(0,R)$, we can use Lemma
        \ref{lemma:lemmaprev2bis} with $v(0)$ in combination with the
        fact that
        $\norm{v(0)}_{L^1(\Omega)}\leq\norm{u(T_1)}_{L^1(\Omega)}\leq\norm{u_0}_{L^1(\Omega)}$
        to obtain that, for some $T_2\geq 0$, 
	\begin{equation}
		\label{eqn:thmasyeq2bisq}
		\norm{v(t)-M_{v_0} G(\cdot, t)}_{L^1(\Omega)}\leq
                \varepsilon \norm{u_0}_{L^1(\Omega)} \qquad \forall
                t\geq T_2 
	\end{equation}
	where $M_{v_0}=\int_\Omega v(0)$. Combining \eqref{eqn:thmasyeq1bisbis} and \eqref{eqn:thmasyeq2bisq} we obtain
	\begin{equation}
		\label{eqn:thmasyeq2bisbisqbis}
		\norm{v(t)-m_{u_0} G(\cdot, t)}_{L^1(\Omega)}\leq 2\varepsilon + \varepsilon \norm{u_0}_{L^1(\Omega)} \qquad \forall t\geq T_2.
	\end{equation}
	Combining now \eqref{eqn:thmasyeq2bisbisqbis} with \eqref{eqn:thmasyeq-1bis} we obtain
	\begin{equation}
	\label{eqn:thmasyeq2bisbisq}
	\norm{u(t+T_1)-m_{u_0} G(\cdot ,t)}_{L^1(\Omega)}\leq
        3\varepsilon+\varepsilon \norm{u_0}_{L^1(\Omega)} \qquad
        \forall t\geq T_2. 
      \end{equation}

To conclude we  need to prove that $G(\cdot,  t)$ and
$G(\cdot,  t+T_1)$ are close in $L^1(\Omega)$ for large
times. This is done in  Lemma \ref{lemma:previo} below. Then  from Lemma
\ref{lemma:previo} with $d=T_1$, we obtain
that there exists a $T_3\geq T_2$ such that 
	\begin{equation}
		\norm{G(\cdot ,t+T_1)-G(\cdot ,t)}_{L^1(\Omega)}\leq \varepsilon \qquad \forall t\geq T_3,
	\end{equation}
	so then
	\begin{equation}
		\label{eqn:thmasyeq3bis}
		\norm{m_{u_0} G(\cdot ,t+T_1)-m_{u_0}G(\cdot ,t)}_{L^1(\Omega)}\leq
                \varepsilon |m_{u_0}|  \qquad \forall t\geq T_3. 
	\end{equation}
	Therefore, combining \eqref{eqn:thmasyeq2bisbisq} and \eqref{eqn:thmasyeq3bis} and denoting $T=T_1+T_3$,
	\begin{equation}
		\label{eqn:thmasyeq5bis}
		\norm{u(\cdot,t)-m_{u_0}G(\cdot ,t)}_{L^1(\Omega)}\leq
                3\varepsilon+\varepsilon |m_{u_0}|+\varepsilon
                \norm{u_0}_{L^1(\Omega)} \qquad \forall t\geq T. 
	\end{equation}
As $\varepsilon>0$ was arbitrary, the theorem is proved.
\end{proof}

Here we prove  the lemma used in  the previous proof.

\begin{lemma}
	\label{lemma:previo}
The Gaussian $G(x,t)= \frac{e^{-\frac{\abs{x}^2}{4t}}}{(4\pi
  t)^{\frac{N}{2}}}$ satisfies, for any $d>0$, 
	\begin{equation}
		\lim_{t\to\infty}\int_{\mathbb{R}^N}\abs{G(x ,t)-G(x ,t+d)}dx=0.
	\end{equation}

\end{lemma}
\begin{proof}
	Adding and subtracting  $\frac{e^{-\frac{\abs{x}^2}{4(t+d)}}}{(4\pi t)^{N/2}}$,
	\begin{equation}
		\begin{aligned}
			&\int_{\mathbb{R}^N}\abs{G(x ,t)-G(x ,t+d)}dx=\int_{\mathbb{R}^N}\abs{\frac{e^{-\frac{\abs{x}^2}{4t}}}{(4\pi t)^{N/2}}-\frac{e^{-\frac{\abs{x}^2}{4(t+d)}}}{(4\pi(t+d))^{N/2}}}dx \\
			&\leq \int_{\mathbb{R}^N}\abs{\frac{e^{-\frac{\abs{x}^2}{4t}}}{(4\pi t)^{N/2}}-\frac{e^{-\frac{\abs{x}^2}{4(t+d)}}}{(4\pi t)^{N/2}}}dx+\int_{\mathbb{R}^N}\abs{\frac{e^{-\frac{\abs{x}^2}{4(t+d)}}}{(4\pi t)^{N/2}}-\frac{e^{-\frac{\abs{x}^2}{4(t+d)}}}{(4\pi (t+d))^{N/2}}}dx \\
			&=\int_{\mathbb{R}^N}\left(\frac{e^{-\frac{\abs{x}^2}{4(t+d)}}}{(4\pi t)^{N/2}}-\frac{e^{-\frac{\abs{x}^2}{4t}}}{(4\pi t)^{N/2}}\right)dx+\int_{\mathbb{R}^N}\left(\frac{e^{-\frac{\abs{x}^2}{4(t+d)}}}{(4\pi t)^{N/2}}-\frac{e^{-\frac{\abs{x}^2}{4(t+d)}}}{(4\pi (t+d))^{N/2}}\right) dx.
		\end{aligned}
	\end{equation}

	Hence, using that $\int_\RN e^{-\frac{\abs{x}^2}{a}}=(\pi a)^{N/2}$, we obtain
$$
=2\left(\left(\frac{t+d}{t}\right)^{N/2}-1\right)\to 0
$$
	when $t\to \infty$, which concludes the proof.
	
\end{proof}

The following theorem demonstrates the optimality of Theorem
\ref{thm:asymL1Rbis}. Its proof is inspired by 
\cite{souplet1999geometry}, where Dirichlet boundary conditions are
considered. 

\begin{theorem}
	\label{thm:optimal_rate_L1}
Let $g:[0,\infty) \to (0,1]$ a monotonically decreasing continuous function such that
$\lim_{t\to\infty}g(t)=0$. Then, there exist an initial value $0\leq u_0\in
L^1(\Omega)$ with $\norm{u_0}_{L^1(\Omega)}=1$ and $T>0$ such that, 
	\begin{equation}
		\label{eqn:opteq1bis}
		\norm{S^\theta(t) u_0-m_{u_0}G(\cdot ,t)}_{L^{1}(\Omega)}\geq g(t) \qquad \forall t>T.
	\end{equation}
\end{theorem}
\begin{proof}
We first consider $\lambda>0$ the first Dirichlet eigenvalue for the
Laplacian operator in the unit ball $B$  and its associated positive
eigenfunction $\psi\geq 0$, normalized such that
$\norm{\psi}_{L^1(B)}=1$. 
Then
$\left(\frac{\lambda}{R^2}, R^{-N} \psi\left(\frac{\cdot}{R}\right)\right)$ is 
an eigenvalue-eigenfunction pair of $\Lap$ in $B(0,R)$ with
homogeneous Dirichlet boundary conditions normalized with 
$L^1(B(0,R))$-norm equal to $1$.

	Now we choose $t_n\to \infty$ such that $g(t_n)=\frac{1}{2^{n+2}}$. Then we consider the following initial datum made up by rescaled copies of $\psi$ in disjoint balls with large radius and far away centres:
	\begin{equation}
		u_0(x)=\sum_{n=1}^\infty\frac{1}{2^n}R_n^{-N}\psi(\frac{x-x_n}{R_n})\chi_{B(x_n,R_n)}(x)
                \geq 0, 
	\end{equation}
	where $\chi_{B(x_n,R_n)}$ is the characteristic function of the $B(x_n,R_n)$ and $R_n>1$ and $x_n$ are chosen so that
	\begin{enumerate}
        \item
          $e^{-\frac{\lambda}{R_n^2}t_{n+1}}\geq 3/4$ (this is
                  possible taking  $R_n$ large enough)

		\item
                  $B(x_n,R_n)\subset \Omega$ with $|x_{n}|>R_{n}+1$ (this  is possible taking $\abs{x_n}$ large enough)

		\item
                  $e^{-\frac{\abs{x_n}-R_n}{4 t_{n+1}}}\leq
                  \frac{(4\pi t_n)^{N/2}}{2^{n+2}\abs{B(0,R_n)}}$
                  (this  is possible taking $\abs{x_n}$ large
                  enough). 
                \end{enumerate}

	Therefore, $\norm{u_0}_{L^1(\Omega)}=1$ and for $x\in
        B(x_n,R_n)$ and $t\in[t_n,t_{n+1}]$ we have 
        \begin{equation}
        	\label{eqn:opteq4bis}
        	\begin{aligned}
        		u(x, t) &  = (S^\theta(t)u_0)(x) 
        		\mygeq{~\eqref{eq:comparison_kernels_theta}}
        		(S^0(t)u_0)(x) \mygeq{} \left(S^0(t)\frac{1}{2^nR_n^{N}}   \psi(\frac{\cdot-x_n}{R_n})\chi_{B(x_n,R_n)}\right)(x) \\
        		& \mygeq{Thm. \ref{thm:compdom}}
                        \frac{S^0_{B(x_n,R_n)}(t)\psi(\frac{\cdot-x_n}{R_n})}{2^nR_n^N}(x)
                        =
                        e^{-\frac{\lambda}{R_n^2}t}\frac{\psi(x)}{2^{n}R_n^N}\geq
                        e^{-\frac{\lambda}{R_n^2} t_{n+1}}\frac{\psi(x)}{2^{n}R_n^N}
        	\end{aligned}
		\end{equation}
	where $S^0_{B(x_n,R_n)}(t_n)$ above is the heat semigroup in the ball
	$B(x_n,R_n)$ with Dirichlet boundary conditions. So, in particular, for any $t\in [t_n,t_{n+1}]$
	\begin{equation}
		\label{eqn:opteq2bis}
		\int_\Omega u(x,t)dx 
			\mygeq{\eqref{eqn:opteq4bis}}
			 \frac{e^{-\frac{\lambda}{R_n^2}
                          t_{n+1}}}{2^n} \mygeq{(i)} \frac{3}{2^{n+2}} 
                    \end{equation}
                    
In addition,  from (ii), for $x\in B(x_n,R_n)$ we have $|x|^{2} \geq |x| \geq
|x_{n}|-R_{n}>1$ and using (iii), for any $t\in [t_n,t_{n+1}]$ we obtain
	\begin{equation}
		\label{eqn:opteq3bis}
	\int_{B(x_n,R_n)}G(x,t)dx \leq
        \int_{B(x_n,R_n)}\frac{e^{-\frac{\abs{x}^2}{4 t_{n+1}}}}{(4\pi
          t_n)^{N/2}}dx \leq
        \int_{B(x_n,R_n)}\frac{e^{-\frac{\abs{x_n} -R_n}{4
              t_{n+1}}}}{(4\pi
          t_n)^{N/2}}dx\myleq{(iii)}\frac{1}{2^{n+2}} . 
	\end{equation}

	Therefore, using  $0\leq m_{u_{0}}\leq 1$ as well as
        \eqref{eqn:opteq2bis} and \eqref{eqn:opteq3bis} we get 
	\begin{equation}
		\norm{u(\cdot,t)-m_{u_0}G(\cdot
                  ,t)}_{L^1(\Omega)}\geq \int_{B(x_n,R_n)}
                (u(x,t)-G(x ,t))dx\geq
                \frac{3}{2^{n+2}}-\frac{1}{2^{n+2}}=\frac{1}{2^{n+1}}
                \geq g(t_{n-1})\geq g(t)
	\end{equation}
which proves  \eqref{eqn:opteq1bis}  for every $t\geq T\defeq t_1$.
      \end{proof}

\section{Asymptotic behavior in the $L^{\infty}(\Omega)$ norm}
\label{sec:asympt-behav-Linftynorm}

In this  section we will focus on obtaining an
asymptotic result in $L^\infty(\Omega)$ for the solution,  for initial data in $u_0\in
L^1(\Omega)$ as in \eqref{eqn:jlvazquez2}, by using the  technique of
\emph{matching asymptotics}. For this,  first we 
prove the result in some time-dependent domain  \textit{far away from
  the hole}, see Theorem \ref{thm:asymfarR}.  Then later we will prove
the result in a complementary time-dependent domain  \textit{near the
  hole}, see Theorem \ref{thm:asymcloseR}. Combining both we will obtain the main Theorem
\ref{cor:completelinfRbis}.  This
procedure is similar to and inspired by the arguments in 
\cite{quiros2007}, \cite{quirosnonlocal}.

So we start with the analysis when we look far away from the hole, that is, when
$\abs{x}^2$ is comparable to or greater than $t$. To do this, we will
use some rescaling and compactness arguments based on
\cite{vazquez2017asymptotic}. The estimates obtained in Theorem
\ref{thm:est} will be crucial for this. The main
result is the one that follows. Notice that as the estimate is far from the
hole there is no reflection in the estimate of the $\theta$-boundary
condition other than the asymptotic mass of the solution.

\begin{theorem} [{\bf Behaviour far from the hole}]
	\label{thm:asymfarR}
Let $N\geq 3$, $u_0\in L^1(\Omega)$ and $u(x,t)=S^\theta(t)u_0(x)$ the solution of
the heat equation with some homogeneous $\theta$-boundary conditions on
$\partial\Omega$.  Then, for any  $\delta>0$, 
	\begin{equation}
		\label{eqn:asymfarRdesiredresult}
		\lim_{t\to\infty}t^{\frac{N}{2}}\norm{u(\cdot,t)-m_{u_0}G(\cdot ,t)}_{L^\infty(\{\abs{x}^2\geq \delta t\})}=0,
	\end{equation}
where $m_{u_0}=\int_\Omega \Phi^\theta(x)u_0(x)dx$ is the asymptotic
mass. Hence, the solution behaves far away from the hole as a Gaussian
times the asymptotic mass of the solution.

In addition, we have convergence of the derivatives, that is for  any multi-index $\alpha$,
\begin{equation}
	\label{eqn:asymfarRdesiredresult2}
	\lim_{t \to\infty} t^{\frac{N+\abs{\alpha}}{2}} \norm{D^{\alpha}
		u(\cdot,t) - m_{u_0} D^{\alpha} G(\cdot
                ,t)}_{L^\infty(\{\abs{x}^2\geq \delta t\})}=0 . 
\end{equation}
\end{theorem}
\begin{proof}
We follow the steps of \cite{vazquez2017asymptotic}. We firstly assume
that $u_0$ is positive and $u_0\in C^\infty_c(\Omega)$. Recall that we
also assume that $0\in \mathring{\hole}$.  
	
	\textbf{Step 1:} We define, for $\lambda>0$ 
	 \begin{equation}
	\label{eqn:rescl1}
	u_{\lambda}(x,t)\defeq \lambda^{N}u(\lambda x, \lambda^2 t).
	\end{equation}
It is straightforward to check that $(u_\lambda)_t-\Lap u_\lambda=0$
in $\frac{\Omega}{\lambda}\times(0,\infty)$ and
$u_\lambda(x,0)=\lambda^Nu_0(\lambda x)$ for $x\in
\frac{\Omega}{\lambda}$. 
Notice that as we take $\lambda>0$ large, we concentrate the hole to a
point. 

\textbf{Step 2:} 
Here we use the estimates from Theorem \ref{thm:est} to obtain uniform
bounds for $u_\lambda$ and its derivatives for $\lambda$ large.

First, take $\delta_{1}>0$ and $t\geq \delta_1$. Then for all $x\in
\frac{\Omega}{\lambda}$ we have from \eqref{eqn:LpLq_estimates_theta} 
	\begin{equation}
		\abs{u_\lambda(x,t)}=\abs{\lambda^N u(\lambda x,\lambda^2 t)}\myleq{\eqref{eqn:LpLq_estimates_theta}}
                \frac{C\norm{u_0}_{L^1(\Omega)}}{\delta_1^{N/2}
                } \qquad \forall t\geq \delta_1. 
	\end{equation}

Second, take $\delta_2>0$ and  we can find $M>1$ such that $\hole\subset
B(0,\frac{M\delta_2}{2})$. Hence, if 
$\abs{x}\geq \delta_2$ and $\lambda \geq M$, then  $|\lambda x|\geq
M \delta_{2}$ and then $\lambda x \in
\Omega$. 

Thus,  for every $t\geq \delta_1$, $\abs{x}\geq \delta_2$ and $\lambda
\geq M$, we get from Theorem \ref{thm:est}
\begin{equation*}
		\abs{\frac{\partial^{k+\abs{\beta}}}{\partial
                    t^{k}\partial x^\beta} u_\lambda(x,t)} =
                \lambda^{N+\abs{\beta}+2k}
                \abs{\frac{\partial^{k+\abs{\beta}}}{\partial
                    t^{k}\partial x^\beta} u(\lambda x,\lambda^2 t)}
                \myleq{Thm \ref{thm:est}}
                \frac{\lambda^{N+\abs{\beta}+2k}
                  C_{\beta,k}\norm{u_0}_{L^1(\Omega)}}{(\lambda^{2}
                  t)^{N/2+k}\min
                  (\lambda t^{1/2}, d(\lambda x,\partial    \Omega))^{\abs{\beta}}}  
	\end{equation*}
	\begin{equation}
		\label{eqn:uinftyboundeq1bis}
		\myleq{$(t\geq
                  \delta_1)$}\frac{\lambda^{\abs{\beta}}C_{\beta,k}\norm{u_0}_{L^1(\Omega)}}{\delta_1^{N/2+k}\min
                  (\lambda\delta_1^{1/2}, d(\lambda x,\partial \Omega))^{\abs{\beta}} }.
	\end{equation}
	Now we use that $d(\lambda x, \partial \Omega)\geq  d(\lambda x,\partial B(0,\frac{M\delta_2}{2}))\geq \lambda\abs{x}-\frac{M\delta_2}{2}\geq \frac{\lambda \abs{x}}{2}$, so
\begin{equation}
\label{eqn:uinftyboundeq2bis}
\min(\lambda \delta_1^{1/2}, d(\lambda x, \partial \Omega)) \geq \min
(\lambda \delta_1^{1/2}, \lambda \abs{x}/2)\geq \lambda\min
(\delta_1^{1/2}, M \delta_2/2).
\end{equation}
Hence, combining \eqref{eqn:uinftyboundeq1bis} and
\eqref{eqn:uinftyboundeq2bis} we obtain 
\begin{equation}
\abs{\frac{\partial^{k+\abs{\beta}}}{\partial t^{k}\partial x^\beta}
  u_\lambda(x,t)} \leq
\frac{C_{\beta,k}\norm{u_0}_{L^1(\Omega)}}{\delta_1^{N/2+k}\min
  (\delta_1^{1/2}, M \delta_2/2)^{\abs{\beta}}}.
\end{equation}
So, we have uniform estimates of the derivatives of $u_\lambda$ for
$t\geq \delta_1$ and $\abs{x}\geq \delta_2$ where
$\delta_1,\delta_2>0$ are arbitrary and $\lambda$ sufficiently large.

\textbf{Step 3:}
        From the  uniform estimates of $u_\lambda$ and its
derivatives,  we have uniform convergence on compact sets of 
$\RN\backslash\{0\}\times(0,\infty)$ of a subsequence and all its
derivatives to a limit function  $u_\lambda\to u_\infty$. Furthermore,
as the derivatives converge, we also have that $(u_\infty)_t-\Lap
u_\infty=0$ in $\RN\backslash\{0\} \times(0,\infty)$. 
	
	\textbf{Step 4:}	
	As we assumed that $u_0$ is positive and $u_0\in
        C^\infty_c(\Omega)$, there exists a $M>0$ such that,
        extending $u_{0}$ by zero outside $\Omega$, 
	\begin{equation}
		MG(x ,1)\geq u_0(x)  \qquad \forall x\in\R^{N}.
	\end{equation}
	Hence, by time monotonicity of the heat semigroup in $\R^{N}$ 
	\begin{equation}
		\label{eqn:umenorkbis}
		MG(x ,t+1)\geq u_{\RN}(x,t) \qquad \forall (x,t)\in\Omega\times[0,\infty).
	\end{equation}
where $u_\RN$ is the solution of the heat equation in $\RN$ with initial datum $u_0$
extended by zero outside $\Omega$. Using then \eqref{eqn:gyryabound},
	\begin{equation}
		u(x,t)\myleq{\eqref{eqn:gyryabound}}
                C\int_{\Omega}\frac{e^{-\frac{\abs{x-y}^2}{4ct}}}{(4\pi
                  t)^{N/2}}u_0(y)dy\leq
                Cu_{\RN}(x,ct)\myleq{\eqref{eqn:umenorkbis}} C M G(x , c(t+1)) \qquad \forall (x,t)\in\Omega\times(0,\infty).
	\end{equation}
	Then, using the self-similarity of $G$, that is,
        $G(x,t)=\lambda^{N}G(\lambda x,\lambda^2 t)$, for $x \in
        \frac{\Omega}{\lambda}$ and $t>0$ we have 
	\begin{equation}
		\label{eqn:uinftyboundbisbis}
		u_\lambda(x,t)=\lambda^Nu(\lambda x,\lambda^2 t)\leq
               C M\lambda^N G(\lambda x,c (\lambda^2 t+1)) =CM G(x,ct+\frac{c}{\lambda^2}).
	\end{equation}
		So, passing to the limit when $\lambda \to \infty$,
	\begin{equation}
		\label{eqn:uinftyboundbis2bis}
		u_\infty(x,t)\leq CM G(x,ct) \qquad \forall (x,t)\in \RN\backslash\{0\}\times(0,\infty).
	\end{equation}
	
\textbf{Step 5:} As $u_\infty$ is a bounded solution of the heat
equation in $\RN\backslash\{0\}$ with $N\geq 3$, we can use Theorem
\ref{thm:removablesing} to remove the singularity to get  that
$u_\infty$ is  in fact   a solution of the heat equation in
$\R^{N}\times (0,\infty)$.

\textbf{Step 6: } In this step we are going to identify  the initial
datum of $u_\infty$. First of all, we know from
\eqref{eq:asymptotic_mass} that 
	\begin{equation}
		\lim_{t\to\infty}\int_\Omega u(x,t)dx=\int_\Omega \Phi^\theta(x)u_0(x)dx = m_{u_0}.
	\end{equation}
	Then, for $t>0$,
	\begin{equation}
		\label{eqn:uinftyboundbis3bis}
		\int_{\frac{\Omega}{\lambda}} u_\lambda(x,t)dx=\int_{\frac{\Omega}{\lambda}} \lambda^N u(\lambda x,\lambda^2 t)dx= \int_\Omega u(x,\lambda^2 t)dx \to m_{u_0}
	\end{equation}
when $\lambda\to\infty$. 
	
	Therefore, from the  uniform convergence of $\{u_\lambda\}$ on
        compact sets of $\RN\setminus\{0\}$ and  the bound
        \eqref{eqn:uinftyboundbisbis} then Lebesgue's theorem gives
        that  for $t>0$, we have  
	\begin{equation}
		\label{eqn:massofuinfty2bis}
		\int_{\RN}u_\infty(x,t)dx=\lim_{\lambda\to\infty}
                \int_{\frac{\Omega}{\lambda}} u_\lambda(x,t)dx
                \myeq{\eqref{eqn:uinftyboundbis3bis}} m_{u_0} 
	\end{equation}
which, in particular, gives that 
        $\{u_\infty(t)\}_{t\geq 0}\subset L^1(\RN)$ is bounded. 
        
Hence, for any sequence $t_{n}\to 0$, we can find a 
subsequence (that we denote the same) such that  
$u_\infty(t_n)\overset{\ast}{\rightharpoonup} \mu$ weakly 
in the sense of measures to a bounded  measure $\mu$, 
that is 
	\begin{equation}
		\lim_{t_n\to 0}\int_\RN u_\infty(x,t_n)\varphi(x)dx=
                \int_\RN\varphi d\mu \qquad \forall \varphi\in
                C_0(\RN) . 
	\end{equation}
Estimate  \eqref{eqn:uinftyboundbis2bis} guarantees that $\mu$ is
concentrated in $0$, so it is a Dirac distribution whose  mass is
determined by  $\eqref{eqn:massofuinfty2bis}$, and therefore  $\mu =
m_{u_0}\delta$.  Since the limit is independent of the weakly
convergent subsequence then the whole sequence 
converges to $\mu$. 
From  the uniqueness of bounded solutions for the heat equation with
bounded measures as initial data (see for example \cite{aroberrobinson1}
Theorem 4.1), we have that 
\begin{equation}
u_\infty(x,t)=m_{u_0}G(x ,t) \qquad (x,t)\in \RN \times(0,\infty).
\end{equation}	

In particular, the limit function $u_{\infty}$ in Step 3 is
independent of the subsequence of $\{u_\lambda\}$ and therefore the
whole family  $\{u_\lambda\}$  converges to $u_\infty$.

\textbf{Step 7:} Now we obtain
 \eqref{eqn:asymfarRdesiredresult}. From the uniform  convergence 
$u_\lambda\to u_\infty$ in compact sets of
$\RN\backslash\{0\}\times(0,\infty)$,  for $t=1$ and  $\delta_1\leq \abs{x}\leq \delta_2$
we have 
	\begin{equation}
		\lim_{\lambda\to\infty}\norm{u_\lambda(\cdot,1)-m_{u_0}G(\cdot ,1)}_{L^\infty(\{\delta_1\leq \abs{x}\leq \delta_2\})}= 0.
	\end{equation}
But from \eqref{eqn:uinftyboundbisbis}, 
$\{u_\lambda\}_\lambda$ are uniformly small for $\abs{x}\geq \delta_2$
and  $\lambda\geq 1$  as they decay exponentially since 
\begin{equation}
	u_\lambda(x,1) \myleq{\eqref{eqn:uinftyboundbisbis}} CM G(x,c+\frac{c}{\lambda^2})= \frac{M}{(4\pi c( 1+\frac{1}{\lambda^2}))^{N/2}}e^{-\frac{\abs{x}^2}{4c(1+\frac{1}{\lambda^2})}}\leq Ce^{-\frac{\abs{x}^2}{C}},
\end{equation}
	for $\lambda\geq 1$. Thus,  we have  uniform convergence for
        $\abs{x}\geq \delta_1$ and 
        \begin{equation}
		\lim_{\lambda\to\infty}\norm{u_\lambda(\cdot,1)-m_{u_0}G(\cdot ,1)}_{L^\infty(\{\delta_1\leq \abs{x}\})}= 0,
	\end{equation}
which, rewritten in terms of the definition of $u_\lambda$,  is
\begin{equation}
			\lim_{\lambda\to\infty}\norm{\lambda^N
                          u(\lambda \cdot,\lambda^2)-m_{u_0} G(\cdot ,1)}_{L^\infty(\{\delta_1\leq \abs{x}\})}=0,
	\end{equation}
and using the self-similarity of $G$ this gives 
	\begin{equation}
		 \lim_{\lambda\to\infty}\lambda^N\norm{u(\lambda
                   \cdot,\lambda^2)-m_{u_0} G(\lambda \cdot, \lambda^2)}_{L^\infty(\{\delta_1\leq \abs{x}\}) }=0.
	\end{equation}	
	Then, renaming $t=\lambda^{2}$ and $y=\lambda x = t^{1/2}x$,
        we get 
	\begin{equation}
		\lim_{t \to\infty} t^{N/2} \norm{u(\cdot,t)-m_{u_0}G(\cdot ,t)}_{L^\infty(\{\delta_1\leq \abs{y}t^{-1/2}\}) }=0,
	\end{equation}
which is \eqref{eqn:asymfarRdesiredresult}.

\textbf{Step 8:} 	If $0 \leq u_0\in L^{1}(\Omega)$, given $\eps
>0$,  we can
consider an approximation $0\leq u_0^\varepsilon\in C^\infty_c(\Omega)$ such
that $\norm{u_0-u_0^\varepsilon}_{L^1(\Omega)}\leq \varepsilon$. Then,
Corollary \ref{cor:LpLq_estimates} gives
\begin{equation}
	\label{eqn:asymfarReq2}
		t^{N/2}\norm{S^\theta(t)(u_0-u_0^\varepsilon)}_{L^\infty(\Omega)}\myleq{~\eqref{eqn:LpLq_estimates_theta}}
                C\norm{u_0-u_0^\varepsilon}_{L^1(\Omega)}\leq
                C\varepsilon , \quad t>0. 
	\end{equation} 
		Furthermore, 
	\begin{equation}
		\label{eqn:asymfarReq1}
		\abs{m_{u_0}-m_{u_0^\varepsilon}} =  \abs{\int_\Omega
                  \Phi^\theta (u_0-u_0^\varepsilon)}\leq\int_\Omega
                \abs{u_0-u_0^\varepsilon}\leq \varepsilon  , \quad t>0
	\end{equation} 
and therefore 
	\begin{equation}
		\label{eqn:asymfarReq3}
		t^{N/2}\norm{(m_{u_0}-m_{u_0^\varepsilon})G(\cdot
                  ,t)}_{L^\infty(\Omega)}\leq C\abs{
                  m_{u_0}-m_{u_0^\varepsilon}} \leq C\varepsilon ,
                \quad t>0. 
	\end{equation}
So using \eqref{eqn:asymfarReq2} and \eqref{eqn:asymfarReq3} and
adding and subtracting $u^\varepsilon(\cdot,t) = S^\theta(t)u_0^\varepsilon$ and
$m_{u_0^\varepsilon} G(\cdot  ,t)$ we get 
	\begin{equation}
		\lim_{t\to\infty}t^{\frac{N}{2}}\norm{u(\cdot,t)-m_{u_0}G(\cdot
                  ,t)}_{L^\infty(\{\abs{x}^2\geq \delta t\})}\leq
                \lim_{t\to\infty}t^{\frac{N}{2}}\norm{u^\varepsilon(\cdot,t)-m_{u^\varepsilon_0}G(\cdot
                  ,t)}_{L^\infty(\{\abs{x}^2\geq \delta
                  t\})}+2C\varepsilon = 2C\varepsilon. 
	\end{equation}
Since  $\varepsilon$ is  arbitrary, we have
\eqref{eqn:asymfarRdesiredresult}. 
	
\textbf{Step 9:}  The case in which $u_0$ is not positive just follows by the
decomposition in its negative and positive part. Using  $u_0=u_0^+-u_0^-$ and
apply the Theorem individually for each $u_0^{\pm}$. Note that
$m_{u_0}=m_{u_0^+}-m_{u_0^-}$ by 
\eqref{eq:asymptotic_mass}. Then, using the triangle inequality of the
norm we obtain the  result as with 
$u(\cdot,t) = S^\theta(t)u_0$ we have 
	\begin{equation}
		\begin{aligned}
			& t^{\frac{N}{2}}\norm{u(\cdot,t) -
                          m_{u_0}G(\cdot
                          ,t)}_{L^\infty(\{\abs{x}^2\geq \delta t\})}=
                        t^{\frac{N}{2}}\norm{S^\theta(t)u_0^+-S^\theta(t)u_0^-
                          -\left(m_{u_0^+}-m_{u_0^-}\right)G(\cdot
                          ,t)}_{L^\infty(\{\abs{x}^2\geq \delta t\})} 
			\\
			& \leq t^{\frac{N}{2}}\norm{S^\theta(t)u_0^+-m_{u_0^+}G(\cdot ,t)}_{L^\infty(\{\abs{x}^2\geq \delta t\})} + t^{\frac{N}{2}}\norm{S^\theta(t)u_0^--m_{u_0^-}G(\cdot ,t)}_{L^\infty(\{\abs{x}^2\geq \delta t\})} \to 0.
		\end{aligned}		
	\end{equation}
\textbf{Step 10: }Now we obtain the convergence of the derivatives. Let us proceed by induction on the order of the multi-index $\alpha$. Assume we have, for any $\delta\geq 0$,
\begin{equation}
	\lim_{t \to\infty} t^{\frac{N+\abs{\alpha}}{2}} \norm{D^{\alpha}
		u(\cdot,t) - m_{u_0} D^{\alpha} G(\cdot ,t)}_{L^\infty(\{\abs{x}^2\geq \delta t\})}=0.
\end{equation}
Let us prove that
\begin{equation}
	\lim_{t \to\infty} t^{\frac{N+\abs{\alpha}+1}{2}} \norm{\frac{\partial}{\partial x_i}D^{\alpha}
		u(\cdot,t) - m_{u_0} \frac{\partial}{\partial x_i} D^{\alpha} G(\cdot ,t)}_{L^\infty(\{\abs{x}^2\geq \delta t\})}=0.
\end{equation}
Take $\delta>0$ and $\abs{x_0}^2\geq \delta t_0$. As $u - m_{u_0}  G$
is a solution of the heat equation in $\Omega$, we can use Theorem
\ref{thm:schauderest} with $Q=\{(x,t): t\geq t_0/2, \ \abs{x}^2\geq
\frac{\delta}{2}t\}$. Then, we obtain 
\begin{equation}
	d_{(x_0,t_0)}\abs{\frac{\partial}{\partial x_i}D^\alpha u(x_0,t_0) - m_{u_0} \frac{\partial}{\partial x_i} D^{\alpha} G(x_0 ,t_0)}\leq \norm{D^{\alpha}
		u(\cdot,t) - m_{u_0} D^{\alpha} G(\cdot ,t)}_{L^\infty(Q)}
\end{equation}
where $d_{(x_0,t_0)}=\inf\{(\abs{x_0-x}^2+\abs{t_0-t})^{1/2} \ :
\ (x,t)\in \partial Q
\}=\min\left(\sqrt{\frac{t_0}{2}},\sqrt{\frac{\delta}{2}t_0}\right)=C\sqrt{t_0}$. Therefore,
taking the supremum over $(x_0,t_0)$ such that $\abs{x_0}^2 \geq
\delta t_0$, 
\begin{equation}
	\begin{aligned}
		& \lim_{t \to\infty} t^{\frac{N+\abs{\alpha}+1}{2}} \norm{\frac{\partial}{\partial x_i}D^{\alpha}
			u(\cdot,t) - m_{u_0} \frac{\partial}{\partial x_i} D^{\alpha} G(\cdot ,t)}_{L^\infty(\{\abs{x}^2\geq \delta t\})}  \\
		& \leq C\lim_{t \to\infty} t^{\frac{N+\abs{\alpha}}{2}} \sup_{s\geq t/2}\norm{D^{\alpha}
			u(\cdot,s) - m_{u_0} D^{\alpha} G(\cdot ,s)}_{L^\infty(\{\abs{x}^2\geq \frac{\delta}{2} s\})} \\
		& \leq C\lim_{t\to\infty}\sup_{s\geq t/2}s^{\frac{N+\abs{\alpha}}{2}}\norm{D^{\alpha}
			u(\cdot,s) - m_{u_0} D^{\alpha} G(\cdot ,s)}_{L^\infty(\{\abs{x}^2\geq \frac{\delta}{2} s\})}=0
	\end{aligned}
\end{equation}
by the induction hypothesis.
\end{proof}

\medskip 
Now we will obtain the asymptotic behaviour of the solutions
near the hole. To do this, we will do a comparison argument with
suitable sub and supersolutions. Later we do the matching of the
asymptotic behavior far and near the  hole. 

We first need the following two lemmas that will allows us to
construct sub and super solutions close to the hole. The first one
 is immediate from the expression for  the  Laplacian for a radial
 function.

\begin{lemma}
	\label{lemma:z}
	Let $0<\gamma<1$ and
	$z(x)\defeq \frac{1}{\abs{x}^{\gamma}}$. Then
	\begin{equation}
		\label{eqn:lemmaz0}
		-\Lap z(x)= \gamma(N-2-\gamma)\dfrac{z(x)}{\abs{x}^{2}}.
	\end{equation}
\end{lemma}

\begin{lemma}
	\label{lemma:ZMbis}
Assume $N\geq 3$. Then, for any $0<\gamma<1$, there exists
$\delta,c,m>0$ and a regular  function $Z:\overline{\Omega}\times (0,\infty)
\to \R$ such that,  for  $z(x)=\frac{1}{\abs{x}^\gamma}$ as in Lemma \ref{lemma:z}, we have 
	\begin{enumerate}
		\item $Z(x,t) > 0$ for $x\in \overline{\Omega}$ and $t
                  >0$. 
		\item $\lim_{t\to\infty} t^{\frac{N}{2}}Z(x,t)=0$
                  uniformly in $\overline{\Omega}$. 
		\item $\frac{\partial Z}{\partial
                    n}(x,t)\geq\frac{m}{1+t^{N/2+1}}$ for every $x\in
                  \partial \Omega$ and $t>0$. 
		\item $Z_t(x,t)-\Lap Z(x,t) \geq c
                  t^{-\frac{N+\gamma}{2}} \dfrac{z(x)}{\abs{x}^{2}}$
                  for $(x,t)$ such that $\abs{x}^2\leq \delta t$.
                \end{enumerate}
\end{lemma}
\begin{proof}
Let $\Phi^0$ be the asymptotic profile of $\Omega$ for Dirichlet
boundary conditions and define $\Psi\defeq 1-\Phi^0 \geq 0$.  Since all the
points of the boundary are maximum points for $\Psi$, by  Hopf lemma we
have that $\restr{\frac{\partial \Psi}{\partial n}}{\partial \Omega}
>0$. 
	
Now let $0<\gamma< 1$ and  define
\begin{displaymath}
Z(x,t)=t^{-\frac{N+\gamma}{2}}(z(x)+\kappa \Psi(x)) >0
\end{displaymath}
where $\kappa>0$ is to be chosen below. Then $Z$ satisfies (i) and (ii) since
$t^{\frac{N}{2}}Z(x,t)=t^{-\frac{\gamma}{2}}(z(x)+\kappa\Psi(x))$ and $(z(x)+\kappa\Psi(x))$ is bounded in $\Omega$.

Now, we  can choose $\kappa$ sufficiently large such that
$\frac{\partial}{\partial n}(z(x)+\kappa \Psi(x)) > m>0$ for
$x\in\partial \Omega$. Therefore $\restr{\frac{\partial Z}{\partial
    n}}{\partial \Omega} \geq \frac{m}{1+t^{N/2+1}}$ because $\gamma<
1$, so $Z$ satisfies (iii). 
	
	Now, let us check (iv).  Using Lemma \ref{lemma:z} and the definition of $\Psi$ we have 
	\begin{equation}
		Z_t-\Lap Z = t^{-\frac{N+\gamma}{2}} \left(\gamma(N-2-\gamma)\dfrac{z(x)}{\abs{x}^{2}}-\left(\frac{N+\gamma}{2}\right)\frac{z(x)+\kappa \Psi(x)}{t}\right).
	\end{equation}
	As $\gamma<1$ and $N\geq 3$, we have $C=\gamma(N-2-\gamma)>0$. Therefore, denoting $D=\frac{N+\gamma}{2}$,
	\begin{equation}
		Z_t-\Lap Z = t^{-\frac{N+\gamma}{2}} (C\dfrac{z(x)}{\abs{x}^{2}}-D\frac{z(x)+\kappa \Psi(x)}{t}).
	\end{equation}
	From the estimates on $\Phi^{0}$  of Proposition
        \ref{prop:estpro}, as $\gamma<1$ and $N\geq 3$, then $N-2 \geq
        1 > \gamma$ and we have that there exists a
        $C_2>0$ such that $\Psi(x) \leq C_2 z(x)$ for every $x\in
        \Omega$. Thus, choosing $D_2=D(1+\kappa C_2)$, 
	\begin{equation}
		Z_t-\Lap Z \geq  t^{-\frac{N+\gamma}{2}} (C\dfrac{1}{\abs{x}^{2}}-D_2\frac{1}{t})z(x).
	\end{equation}
	Hence, choosing $\delta>0$ sufficiently small such that
        $c=C-\delta D_2>0$, for $\abs{x}^2\leq \delta t$ we have 
	\begin{equation}
		Z_t-\Lap Z \geq c t^{-\frac{N+\gamma}{2}}
                \dfrac{z(x)}{\abs{x}^{2}}. 
	\end{equation}
\end{proof}

Now we prove the asymptotic result close to the hole. For this we
will use comparison principle in variable domains, Theorem 
\ref{thm:compvarneumann},  with suitable sub and supersolutions
constructed with the help of $Z$ in Lemma \ref{lemma:ZMbis}. Notice
that property (iii) in Lemma \ref{lemma:ZMbis} will be used to cope
with $\theta$--boundary conditions other than Dirichlet. 

\begin{theorem} [{\bf Behaviour near the hole}]
  \label{thm:asymcloseR}
  
Let $N\geq 3$, $u_0\in
L^1(\Omega)$ and $u(x,t)=S^\theta(t)u_0(x)$ the solution of the heat
equation with homogeneous $\theta-$boundary conditions on $\partial
\Omega$. Then, there exists $\delta>0$ such that 
	\begin{equation}
		\lim_{t\to\infty}t^{\frac{N}{2}}\norm{u(\cdot,t)-m_{u_0}\Phi^\theta(\cdot)
                  G( \cdot ,t)}_{L^\infty(\{\abs{x}^2\leq \delta t\})}=0,
	\end{equation}
where $m_{u_0}=\int_\Omega \Phi^\theta(x) u_0(x)dx$ is the asymptotic
mass of $u$. Hence, the solution behaves near the hole as a Gaussian
with the asymptotic mass times the asymptotic profile $\Phi^\theta$. 
\end{theorem}
\begin{proof}
We define 
	\begin{equation}
		\label{thm:asymcloseeqn-2}
		v^+(x,t)\defeq \Phi^\theta (x) G(x, t)+\sigma
                Z(x,t) \geq 0 
	\end{equation}
	where $Z(x,t)$ and $\delta>0$ are as in Lemma
        \ref{lemma:ZMbis} and now we choose  $\gamma<1/2$. Let us see
        that $v^+$ is a supersolution of the heat equation in some
        variable domain. First, from Lemma \ref{lemma:ZMbis} and using
        that $\Phi^{\theta}$ is harmonic and $G$ satisfies the heat
        equation, we get 
	\begin{equation}
		\label{thm:asymcloseeqn-1bis}
		\frac{\partial v^+}{\partial t}-\Lap v^+ \geq c\sigma
                t^{-\frac{N+\gamma}{2}} \dfrac{z(x)}{\abs{x}^{2}} -
                2\abs{\nabla \Phi^\theta}\abs{\nabla G(\cdot, t)}.
	\end{equation}
	Now, from the explicit form of $G(x ,t)$, we have that, for $\abs{x}^2\leq \delta t$,
	\begin{equation}	
		\label{thm:asymcloseeqn1bis}
		\abs{\nabla G(x ,t)}\leq \frac{C}{t^{\frac{N+1}{2}}}.
	\end{equation}
	Furthermore, from the estimates of the asymptotic profile $\Phi^\theta(x)$ of Proposition \ref{prop:estpro} we have
	\begin{equation}
		\label{thm:asymcloseeqn2bis}
		\abs{\nabla \Phi^\theta(x)}\leq \frac{C}{\abs{x}^{N-1}}.
	\end{equation}
	Therefore, combining \eqref{thm:asymcloseeqn-1bis}, \eqref{thm:asymcloseeqn1bis} and \eqref{thm:asymcloseeqn2bis} and using that $\abs{x}^2\leq \delta t$,
	\begin{equation}
		\begin{aligned}
			\frac{\partial v^+}{\partial t}-\Lap v^+ & \geq \frac{c\sigma t^{-\frac{N+\gamma}{2}}}{\abs{x}^{2+\gamma}}-\frac{Ct^{-\frac{N+1}{2}}}{\abs{x}^{N-1}}
			 =t^{-\frac{N+\gamma}{2}}\left(\frac{c\sigma}{\abs{x}^{2+\gamma}}-\frac{C }{t^{\frac{1-\gamma}{2}}\abs{x}^{N-1}}\right) \\
			 & \mygeq{$\abs{x}^2\leq \delta t$} t^{-\frac{N+\gamma}{2}}\left(\frac{c\sigma}{\abs{x}^{2+\gamma}}-\frac{C\delta^{\frac{1-\gamma}{2}}}{\abs{x}^{N-\gamma}}\right).
		\end{aligned}
	\end{equation}
Now, as $\gamma<1/2$, we have $(N-\gamma)-(2+\gamma)=N-2-2\gamma> 0$,
and then, as $0\in \mathring{\hole}$, we have $\abs{x}^{-(2+\gamma)}\geq
C\abs{x}^{-(N-\gamma)}$ for $C>0$ large enough and for all $x\in \Omega$. Therefore, we can
choose $\sigma$ large enough so that 
	\begin{equation}
		\label{thm:asymcloseeqn9}
		\frac{\partial v^+}{\partial t}-\Lap v^+ \geq 0 \qquad \abs{x}^2\leq \delta t, \ t>0.
	\end{equation}
	Therefore, $v^+$ is a supersolution in the region
        $\abs{x}^2\leq \delta t$.

        Now, we will show that we can
        compare $v^+$ and $u$ at $\partial \Omega$ and at 
        $\abs{x}^2=\delta t$ for $t\geq T$ large enough. First, at
        $\partial \Omega$, we prove that, choosing $\sigma>0$ sufficiently
        large, 
	\begin{equation}
		\label{thm:asymcloseeqn6bisbis}
		B_\theta(v^+)(x,t) > 0 = B_\theta(u)(x,t) \qquad
                x\in\partial\Omega, \quad   t>0.
	\end{equation} 
	Indeed,
	\begin{equation}
		\label{thm:asymcloseeqn7}
		B_\theta(v^+)=B_\theta(\Phi^\theta) G(\cdot, 0)
                +\sin(\frac{\pi}{2}\theta)\Phi^\theta\frac{\partial
                  G(\cdot, 0) }{\partial n}+\sigma B_\theta(Z). 
	\end{equation}
	Now, $B_\theta(\Phi^\theta)=0$  and, since  $0\notin \partial \Omega$, 
	\begin{equation}
		\abs{\frac{\partial G(\cdot, 0)}{\partial
                    n}(x,0,t)}\leq \abs{\nabla G(x ,t)} \leq
                \frac{\abs{x}e^{-\frac{\abs{x}^2}{4t}}}{2t(4\pi
                  t)^{N/2}}\leq \frac{C}{1+t^{N/2+1}} \qquad x\in
                \partial \Omega , \quad t>0 . 
	\end{equation}
Then in the Dirichlet part of the boundary where $\theta =0$ we have, by
(i) in Lemma \ref{lemma:ZMbis}, 
$B_\theta(v^+) = \sigma B_\theta(Z) >0$. On the rest of the boundary,
using property (iii) from Lemma \ref{lemma:ZMbis}, we have 
	\begin{equation}
B_\theta(v^+) \geq  \Big( \sigma   \sin(\frac{\pi}{2}\theta) m -C\Big)
\frac{1}{1+t^{N/2+1}} >0  
	\end{equation}
provided  $\sigma$ is large enough, because, as $\partial \Omega$ is compact and $\theta\in C^{1,\alpha}(\partial\Omega)$, $\sin(\frac{\pi}{2}\theta)$ is bounded below by a positive constant in each connected component of $\partial \Omega$ in which we do not have Dirichlet conditions.

Now, let us compare $v^+$ and $u$ at $\abs{x}^2=\delta t$. Given
$\varepsilon>0$, we use  Theorem \ref{thm:asymfarR} with  $\delta$ from
Lemma \ref{lemma:ZMbis}  and then, for sufficiently large $T>0$, 
	\begin{equation}
		\label{thm:asymcloseeqn5bisbis}
		u(x,t)\leq (1+\varepsilon)m_{u_0}G(x ,t) \qquad for \ \abs{x}^2=\delta t, \ \  t\geq T.
	\end{equation}
	Then, as $\Phi^\theta(x)\to 1$ when $\abs{x}\to \infty$ and $\sigma Z(x,t)\geq 0$, choosing $T$ large enough we have that
\begin{equation}
	\label{thm:asymcloseeqn6}
	 u(x,t)<(1+2\varepsilon)m_{u_0} \Phi^\theta (x) G(x,t) \leq
         (1+2\varepsilon)m_{u_0} v^+(x,t)  \qquad for \ \abs{x}^2=\delta t, \
         \  t\geq T
\end{equation} 
for any value of $\sigma>0$.

Finally, to be able to use Theorem \ref{thm:compvarneumann} with $v^+$
and $u$, we need to compare the functions at some fixed time $T$. To
do that, observe that given $T>0$ as above, since $u(\cdot, T)$ is
bounded and $Z$ is strictly positive  if $\abs{x}^2\leq \delta T$, we can then  fix $\sigma$ large enough so that 
	\begin{equation}
		\label{thm:asymcloseeqn4bisbis}
		u(x,T)<(1+2\varepsilon)m_{u_0}v^+(x,T) \qquad \forall \abs{x}^2\leq \delta T.
	\end{equation}

Now, from \eqref{thm:asymcloseeqn9}, \eqref{thm:asymcloseeqn6bisbis},
\eqref{thm:asymcloseeqn6} and  \eqref{thm:asymcloseeqn4bisbis} we can
apply Theorem \ref{thm:compvarneumann} with $t_1=T$, any $t_2\geq t_1$, $\Omega_{[T,t_2]}\defeq \{x\in \Omega, \ t\in[T,t_2]: \abs{x}^2\leq \delta t \}$, $S_2 = \partial\Omega\times[T,t_2]$ and $S_1=\{(x,t) : t\in [T,t_2], \ \abs{x}^2= \delta t\}$.
to obtain 
	\begin{equation}
		\label{thm:asymcloseeqn7bisbis}
		u(x,t)\leq (1+2\varepsilon)m_{u_0}v^+(x,t) \qquad \forall x\in \Omega \ \ \abs{x}^2\leq \delta t \ \ \ \forall t \in [T,t_2].
	\end{equation} 

In addition, as $t_2$ was arbitrary, \eqref{thm:asymcloseeqn7bisbis}
holds for all $t\geq
T$. Therefore, for $\abs{x}^2\leq \delta t$, $t\geq T$,  
	\begin{equation}
		\begin{aligned}
			 t^{\frac{N}{2}} (u(x,t)-& m_{u_0}\Phi^\theta(x)G(x ,t))  \myleq{\eqref{thm:asymcloseeqn7bisbis}}
			t^{\frac{N}{2}} ((1+2\varepsilon)m_{u_0}v^+(x,t)-m_{u_0}\Phi^\theta(x)G(x ,t)) \\
			& \myeq{\eqref{thm:asymcloseeqn-2}}
			 	 t^{\frac{N}{2}}( 2\varepsilon
                         m_{u_0}\Phi^{\theta}(x) G(x  ,t)+(1+2\varepsilon)m_{u_0}\sigma
                         Z(x,t))  \\ 
			& \leq C\varepsilon m_{u_0}+(1+2\varepsilon)\sigma m_{u_0}t^{\frac{N}{2}}Z(x,t)\to C\varepsilon m_{u_0}
		\end{aligned}
	\end{equation}
	when $t\to\infty$ due to (ii) in  Lemma \ref{lemma:ZMbis}. So,
        as $\varepsilon$ was arbitrary, we have 
	\begin{equation}
		\limsup_{t\to\infty} \left( t^{\frac{N}{2}} \sup_{\{\abs{x}^2\leq \delta t\}}\left\{u(x,t)-m_{u_0}\Phi(x)G(x ,t)\right\}\right)\leq 0.
	\end{equation}

A similar argument can be carried out with a subsolution
\begin{displaymath}
v^-(x,t)\defeq \Phi^\theta (x)G(x ,t)-\sigma Z(x,t)
\end{displaymath}
 to obtain the corresponding  result for $\liminf$. This concludes the proof.
\end{proof}

\begin{remark}
Note that in Theorem \ref{thm:asymcloseR} we are not obtaining the
convergence of the derivatives as in
\eqref{eqn:asymfarRdesiredresult2} in Theorem \ref{thm:asymfarR}. In fact, in general, this result
is not true near the hole due to the boundary conditions. For example, if we consider the case of
Neumann boundary conditions, we have $\frac{\partial u}{\partial n}=0$
on $\partial \Omega$ while $\frac{\partial G}{\partial n}$ is of  order
$t^{-\frac{N+1}{2}}$ so \eqref{eqn:asymfarRdesiredresult2} can not be
true near the hole. 
\end{remark}

Now we can prove the main result in this section which is the
analogous to \eqref{eqn:jlvazquez2}. For this, for $N\geq 3$,  we match the results in
Theorems \ref{thm:asymfarR} 
    and \ref{thm:asymcloseR}.  Observe that here both the
asymptotic mass of the solution and the asymptotic profile of the
problem intervene. 

	\begin{theorem} [{\bf Convergence in the sup norm}]
	\label{cor:completelinfRbis}
	Let $u_0\in L^1(\Omega)$ and 
        $u(x,t)=S^\theta(t)u_0$ be the solution of the heat equation 
 with homogeneous $\theta$-boundary conditions on
 $\partial \Omega$.

Then if $N\geq 3$ or if $N=2$ and
$\theta\not\equiv 1$, that is, except for Neumann boundary conditions, 
	\begin{equation}
		\label{eqn:completelinfRbiseq1}
		\lim_{t\to\infty}t^{\frac{N}{2}}\norm{u(\cdot,t)-m_{u_0}\Phi^\theta (\cdot) G(\cdot ,t)}_{L^\infty(\Omega)}=0,
	\end{equation}
where $m_{u_0}=\int_{\Omega} \Phi^\theta(x)u_0(x)dx$ is the asymptotic
mass. Hence, the solution behaves as a Gaussian times  the asymptotic
mass of the solution and the asymptotic profile. 

\end{theorem}
\begin{proof}
Assume first $N\geq 3$. 
  Taking  $\delta>0$ from Theorem \ref{thm:asymcloseR} in  Theorem
\ref{thm:asymfarR} we get for $\abs{x}^{2}\geq \delta t$, adding and
subtracting $m_{u_0} G(\cdot ,t)$, 
\begin{displaymath}
  t^{\frac{N}{2}}\abs{ u(x,t)-m_{u_0}\Phi^\theta (x) G(x ,t)} \leq
  t^{\frac{N}{2}}\abs{ u(x,t)-m_{u_0} G(x ,t)}  + |m_{u_0}|
  t^{\frac{N}{2}}  G(x ,t) (1-\Phi^\theta (x)) . 
\end{displaymath}
The first term is uniformly small for large times by Theorem
\ref{thm:asymfarR} 
 and
\begin{displaymath}
  |m_{u_0}|   t^{\frac{N}{2}}  G(x ,t) (1-\Phi^\theta (x)) \leq C   |m_{u_0}|   (1-\Phi^\theta (x))
\end{displaymath}
which is as small as we want for all large $t$, since $\abs{x}^{2}\geq
\delta t$ and as $N\geq 3$, $\Phi^\theta(x)\to 1$ as
    $\abs{x}\to \infty$,  see Proposition \ref{prop:estpro}. Then, we just
    combine  Theorems \ref{thm:asymfarR} and \ref{thm:asymcloseR} to
    obtain the  result. 

Now,  if $N=2$ and
$\theta\not\equiv 1$, that is, except for Neumann boundary conditions,
as in Theorem \ref{thm:asymL1Rbis} we have  $\Phi^{\theta} =0$,
$m_{u_{0}}=0$ and $\lim_{t \to \infty} S^{\theta}(t) u_{0} = 0$
  in  $L^{1}(\Omega)$. Then using Corollary \ref{cor:LpLq_estimates}
  and the semigroup property, we get
  \begin{displaymath}
   \big(\frac{t}{2}\big)^{\frac{N}{2}}  \norm{S^\theta(t)u_0}_{L^{\infty}(\Omega)}\leq
     C\norm{S^\theta(\frac{t}{2}) u_0}_{L^{1}(\Omega)} \to 0
  \end{displaymath}
as $t\to \infty$, which proves \eqref{eqn:completelinfRbiseq1}. 
  \end{proof}

The following result shows the optimality of Theorem
\ref{cor:completelinfRbis}.

	\begin{theorem}
          \label{thm:optimal_rate_Linfty}
          
	Let $g:[0,\infty) \to (0,1]$ a monotonically decreasing continuous
	function such that $\lim_{t\to\infty}g(t)=0$. Then, for any homogeneous
	$\theta-$boundary condition,  there exist an
	initial value $u_0\in L^1(\Omega)$, with $\norm{u_0}_{L^1(\Omega)}=1$,
	and a sequence of times $t_n\to \infty$ such that 
	\begin{equation}
		\label{eqn:opteq1}
		t_n^{N/2}\norm{S^\theta(t_n) u_0-m_{u_0}\Phi^\theta(\cdot)G(\cdot ,t_n)}_{L^\infty(\Omega)}\geq g(t_n).
	\end{equation}
\end{theorem}

The proof of the theorem is mainly based on the following general
proposition, based on the Banach-Steinhaus theorem.

\begin{proposition}
\label{prop:presouplet2}

Let $\{T(t)\}_{t\geq 0}$ be a family of linear bounded operators between the
Banach spaces  $X$
and $Y$. Assume that the norms of the operators are
globally  bounded below and locally bounded above, that is,
$0<m\leq \norm{T(t)}_{\mathcal{L}(X,Y)}$ for all $t\geq 0$ and  given
$T_0\geq 0$, there exist  some constant  $M(T_0)$, such that
$\norm{T(t)}_{\mathcal{L}(X,Y)}\leq M(T_0)$ for every $t\leq 
T_0$.

Then, for any continuous	function $g:[0,\infty) \to (0,1]$  such
that $\lim_{t\to\infty}g(t)=0$, there exists $u_0\in X$ such that  
	\begin{equation}
		\label{eqn:presouplet2eq1}
		\limsup_{t\to\infty}
                \frac{\norm{T(t)u_0}_{Y}}{g(t)}=\infty . 
	\end{equation}

\end{proposition}
\begin{proof}
	Let us argue by contradiction. Assume that, for every $u_0\in X$,
	\begin{equation}
		\limsup_{t\to\infty}
                \frac{\norm{T(t)u_0}_{Y}}{g(t)}\leq C_{u_0} . 
\end{equation}
In particular, as $\norm{T(t)}_{\mathcal{L}(X,Y)}$ are locally bounded
in $t\geq 0$ and $g$ is continuous, choosing $C_{u_0}$ larger if
necessary, 
	\begin{equation}
		\frac{\norm{T(t)u_0}_{Y}}{g(t)}\leq C_{u_0} \qquad
                \forall t\in [0,\infty) . 
	\end{equation}
	Then, the uniform boundedness principle implies that there exists a $C>0$ such that
	\begin{equation}
		\frac{\norm{T(t)}_{\mathcal{L}(X,Y)}}{g(t)}\leq C \qquad \forall t\in [0,\infty)
	\end{equation}
	But, then as $\lim_{t\to\infty}g(t)=0$, there exists $T>0$ such that
	\begin{equation}
		\norm{T(t)}_{\mathcal{L}(X,Y)}\leq Cg(T)<m
	\end{equation}
	which is a contradiction.
\end{proof}

\begin{proof}[{\bf Proof of Theorem \ref{thm:optimal_rate_Linfty}}] 
	We consider the linear  operators
	\begin{equation}
		\begin{tabular}{@{}cccc@{}}
			$T(t):$ & $L^1(\Omega)$ & $\longrightarrow$ & $L^\infty(\Omega)$ \\
			& $u_0$ & $\mapsto$ & $t^{\frac{N}{2}}\left(S^\theta(t)u_0-m_{u_0}G(\cdot ,t)\right)$ \\
		\end{tabular}	
\end{equation}
 which are also uniformly bounded because, using Corollary
 \ref{cor:LpLq_estimates} 
\begin{equation}
\norm{T(t)u_0}_{L^\infty(\Omega)}\leq
t^{\frac{N}{2}}\norm{S^\theta(t)u_0}_{L^\infty(\Omega)} +
\norm{u_0}_{L^1(\Omega)}t^{\frac{N}{2}} \norm{G(\cdot
  ,t)}_{L^\infty(\Omega)} \myleq{\eqref{eqn:LpLq_estimates_theta}} C
\norm{u_0}_{L^1(\Omega)} . 
\end{equation}
In addition, let us prove that $\norm{T(t)}\geq c$ for some positive
constant $c>0$ independent of $t$. First, we consider $(\lambda,
\psi)$ the first eigenvalue and eigenfunction of $\Lap$ in the unit
ball $B$ with homogeneous Dirichlet boundary conditions with $\psi \geq 0$ and normalized such that
$\norm{\psi}_{L^1(B)}=1$. Then
$\left(\frac{\lambda}{R^2}, R^{-N} \psi\left(\frac{\cdot}{R}\right)\right)$ is 
an eigenvalue-eigenfunction pair of $\Lap$ in $B(0,R)$ with
homogeneous Dirichlet boundary conditions normalized with 
$L^1(B(0,R))$-norm equal to $1$. Now, for fixed  $t>0$, we choose
$R=\sqrt{t}$ and $x_0\in\Omega$ with $\abs{x_0}$ large enough so that 
\begin{equation}
 t^{N/2} G(x_0 ,t)\leq \frac{e^{-\lambda}\psi(0)}{2}
\end{equation}
Then, if we choose
$u_0(x)=\chi_{B(x_0,R)}R^{-N}\psi\left(\frac{x-x_0}{R}\right) \geq 0$, we have
$\norm{u_0}_{L^1(\Omega)}=1$ and using $0<m_{u_{0}}\leq 1$ and comparison Theorems
\ref{thm:neugeqdir} and \ref{thm:compdom},  we obtain 
\begin{equation}
\begin{aligned}
\norm{T(t)u_0}_{L^\infty(\Omega)} & \geq T(t)u_0(x_0)   \geq
t^{N/2}S^\theta(t)u_0 (x_0) - \frac{e^{-\lambda}\psi(0)}{2} \mygeq{Thm
  \ref{thm:neugeqdir}} t^{N/2}S^0(t)u_0 (x_0) -
\frac{e^{-\lambda}\psi(0)}{2}  \\ 
& \mygeq{Thm \ref{thm:compdom}} t^{N/2}S^0_{B(x_0,R)}(t)u_0 (x_0) -
\frac{e^{-\lambda}\psi(0)}{2} = 
e^{-\lambda}\psi(0)-\frac{e^{-\lambda}\psi(0)}{2}=\frac{e^{-\lambda}\psi(0)}{2} 
\end{aligned}	
\end{equation}
where $S^0_{B(x_0, R)}(t)$ above is the heat semigroup in the ball
$B(x_0, R)$ with Dirichlet boundary conditions. 
Therefore, as $\norm{u_0}_{L^1(\Omega)}=1$, we obtain that
$\norm{T(t)}\geq \frac{e^{-\lambda}\psi(0)}{2}$ for every $t\geq
0$.

Hence, we can use Proposition \ref{prop:presouplet2} to $\{T(t)\}$
to obtain the result. 
\end{proof}

\begin{remark}
  \label{rem:Herraiz_result}
  One of the few results in an exterior domain of which we are
  acquainted with  are those in \cite{Herraiz1998}, which describe the behaviour
  of solutions with homogeneous Dirichlet conditions for initial data
  which behaves like $\abs{x}^{-\alpha}$ as $\abs{x}\to\infty$. In
  particular, the following lemma is stated. 

  \begin{lemma}[\cite{Herraiz1998}, Lemma 3.2 a)]
    \label{lemma:herraiz}
    Let $\Omega\subset \RN$ be a regular exterior domain with
    $N\geq 3$. Let $u: \Omega\times [0,\infty) \to \R$ be a solution
    of
    \begin{equation}
      \left\{
        \begin{aligned}
          & u_t(x,t)-\Lap u(x,t) = 0 \qquad (x,t)\in\Omega \times (0,\infty) \\
          & u(x,t)=0 \qquad x\in\partial \Omega ,\ t\in (0,\infty) \\
          & u(x,0)= u_0(x) ,
        \end{aligned}
      \right. 
    \end{equation}
    where $u_0(x)\sim A\abs{x}^{-\alpha}$ as $\abs{x}\to \infty$ for
    some $A>0$ and $\alpha>N$. Then, when $t\gg 1$,
    \begin{equation}
      \label{eqn:herraizexpected}
      u(x,t)= \frac{\Phi^0(x)}{(4\pi t)^{N/2}} \left(\int_\Omega
        u_0(y)dy\right)  e^{-\frac{\abs{x}^2}{4t}} (1+o(1)) \qquad
      x\in\Omega, \ \abs{x}^2\leq Ct\log(t) 
    \end{equation}
    where $\Phi^0(x)$ is the asymptotic profile for homogeneous
    Dirichlet boundary conditions.
  \end{lemma}

  However, the asymptotic behaviour that Lemma \ref{lemma:herraiz}
  describes seems not to be taking into account the loss of mass
  through the hole as the whole mass $\int_\Omega u_0(y)dy$ appears
  explicitly in the estimate. 
  This can be checked with the explicit radial solution of the heat equation in $\Omega$ with
   Dirichlet boundary conditions when $\Omega\defeq \mathbb{R}^3\backslash B(0,1)$ 
  \begin{equation}
    \label{eqn:herraizcounterexample}
    u(x,t)=e^{-\frac{(\abs{x}-1)^2}{4(t+1)}}\cdot\frac{(\abs{x}-1)}{4\abs{x}(t+1)^{3/2}}
    \qquad x\in \Omega\defeq \mathbb{R}^3\backslash B(0,1), \ \
    t\geq 0 . 
  \end{equation}
  In Figure \ref{fig:com}
  we present a comparison of the exact solution
  \eqref{eqn:herraizcounterexample},  the
  expected asymptotic behaviour in  Theorem
  \ref{cor:completelinfRbis} and the expected behaviour for
  $t=100$ predicted by Lemma \ref{lemma:herraiz},
  \eqref{eqn:herraizexpected}.  The picture depicts the height of the
  functions in terms of the radial coordinate $|x|\geq 1$,  at time $t=100$.
  As the figure shows, the exact solution
  \eqref{eqn:herraizcounterexample} seems not to coincide with the
  behaviour predicted by Lemma \ref{lemma:herraiz}, but rather with the one
  in Theorem \ref{cor:completelinfRbis}.
  \begin{figure}[H]
    \centering \includegraphics[width=0.7\textwidth]{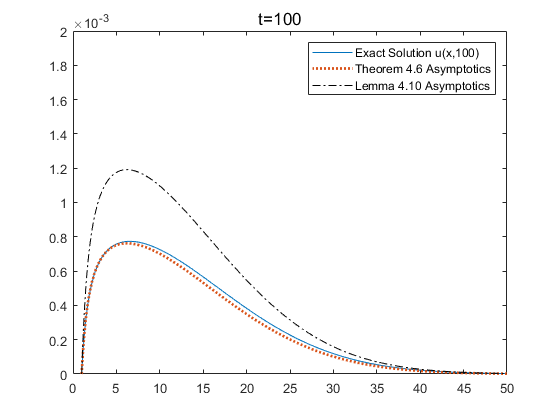}
    \caption{Comparison between  \cite{Herraiz1998} asymptotics and
the one in       Theorem \ref{cor:completelinfRbis}.}
\label{fig:com}
\end{figure}

  Notice that the initial datum for 
  \eqref{eqn:herraizcounterexample} is  
  $u_0(x)=e^{-\frac{(\abs{x}-1)^2}{4}}\cdot\frac{(\abs{x}-1)}{4\abs{x}}$
  which decays faster than the required decay
  $\abs{x}^{-\alpha}$ when $\abs{x}\to\infty$ in Lemma
  \ref{lemma:herraiz}.

  However, a small
  perturbation    $u_0^\varepsilon(x)=u_0(x)+\varepsilon \abs{x}^{-(N+1)}$ is under
  the hypothesis of Lemma \ref{lemma:herraiz}. Hence, Lemma
  \ref{lemma:herraiz} would imply
  that
  $$u_\varepsilon(x,t)=\frac{\Phi^0(x)}{(4\pi
    t)^{N/2}}\left(\int_\Omega u_0^\varepsilon(y)dy\right)
  e^{-\frac{\abs{x}^2}{4t}} (1+o(1))$$ when $t\to\infty$.
  Furthermore, due to \eqref{eqn:LpLq_estimates_theta} we have 
  $$
  \abs{S^{0}(t)u_0^\varepsilon(x)-S^{0}(t)u_0(x)}\leq\norm{u_0^\varepsilon-u_0}_{L^1(\Omega)}t^{-N/2}\leq
  Ct^{-N/2} \varepsilon, \quad x\in \Omega, \ t>0 
  $$
  and  then we would have
$u(x,t)=\frac{\Phi^0(x)}{(4\pi t)^{N/2}}\left(\int_\Omega
  u_0(y)dy\right) e^{-\frac{\abs{x}^2}{4t}} (1+o(1)+O(\varepsilon))$
when $t\to\infty$ and $\abs{x}^2\leq Ct$ which clearly does not happen
as Figure \ref{fig:com} shows.

\end{remark}

\section{Asymptotic behavior in the $L^{p}(\Omega)$ norm}
\label{sec:asympt-behav-Lpnorm}

Finally,  combining Theorem   \ref{cor:completelinfRbis} with Theorem
\ref{thm:asymL1Rbis}, we obtain an asymptotic result in
$L^p(\Omega)$ by interpolation. Observe that this result is the
analogous to Theorem \ref{thm:jlvazquez} in an exterior domain. 

\begin{theorem}
	\label{thm:asympfinal}

Let $u_0\in L^1(\Omega)$ and  $u(x,t)=S^\theta(t)u_0$ be
the solution of the heat equation with homogeneous $\theta$-boundary
conditions on $\partial \Omega$.

Then if $N\geq 3$ or if $N=2$ and
$\theta\not\equiv 1$, that is, except for Neumann boundary conditions,
for  any $1\leq p \leq \infty$, 
\begin{equation}
\lim_{t\to\infty}t^{\frac{N}{2}(1-\frac{1}{p})}\norm{u(\cdot,t)-m_{u_0}\Phi^\theta(\cdot)G(\cdot ,t)}_{L^p(\Omega)}=0,
\end{equation}
where $m_{u_0}=\int_{\Omega} \Phi^\theta(x)u_0(x)dx$ is the asymptotic mass.
\end{theorem}
\begin{proof}
Assume first $N\geq 3$.   First we prove the case $p=1$ by using  Theorem \ref{thm:asymL1Rbis}
  and using that, as  $N\geq 3$ then  $\Phi^\theta(x)\to 1$ when
  $\abs{x}\to \infty$, see  Proposition \ref{prop:estpro}. For this,
  notice that adding and subtracting
  $m_{u_0} G(\cdot ,t)$ 
  \begin{displaymath}
    \norm{u(t)-m_{u_0}\Phi^\theta(\cdot)G(\cdot ,t)}_{L^1(\Omega)}
    \leq      \norm{u(t)-m_{u_0} G(\cdot ,t)}_{L^1(\Omega)} +
    \norm{m_{u_0}(1-\Phi^\theta(\cdot))G(\cdot ,t)}_{L^1(\Omega)} . 
  \end{displaymath}
By Theorem \ref{thm:asymL1Rbis}  the  first term goes to zero as $t\to
\infty$ while for the second, given $\eps>0$ we  choose $R>0$ such that for
$\abs{x}\geq R$, $0\leq 1-\Phi^\theta(x)\leq \varepsilon$. Then 
$G(\cdot, t)$  decays in time  to $0$ in $L^1(\Omega\cap B(0,R))$ just
because of the decay in $L^\infty(\Omega)$ and 
$\Omega\cap B(0,R)$ is a bounded set. Therefore, splitting the
integral for $\abs{x}\leq R$ and $\abs{x}\geq R$ we have 
\begin{displaymath}
   \norm{m_{u_0}(1-\Phi^\theta(\cdot))G(\cdot ,t)}_{L^1(\Omega)} \leq
  | m_{u_0} | \norm{ G(\cdot     ,t)}_{L^1(\Omega\cap\{\abs{x}\leq
    R\})} +     |m_{u_0}|\eps \to  |m_{u_0}|\eps 
\end{displaymath}
as $t\to \infty$. Since $\eps>0$ is arbitrary, we get the result for
$p=1$.

Now we denote $f(x,t)=u(x,t)-m_{u_0}\Phi^\theta(x)G(x ,t)$, so we have
already proved that 
$\lim_{t\to\infty}\norm{f(t)}_{L^1(\Omega)}=0$. In addition, by
Theorem \ref{cor:completelinfRbis} we have
$\lim_{t\to\infty}t^{N/2}\norm{f(t)}_{L^\infty(\Omega)}=0$.
Therefore, using interpolation we get 

	\begin{equation}
		\lim_{t\to\infty}t^{\frac{N}{2}(p-1)}\norm{f(t)}_{L^p(\Omega)}^p\leq \lim_{t\to\infty}\left((t^{\frac{N}{2}}\norm{f(t)}_{L^\infty(\Omega)})^{p-1}\norm{f(t)}_{L^1(\Omega)} \right)=0,
	\end{equation}
	which is the  result.

Now,  if $N=2$ and
$\theta\not\equiv 1$, that is, except for Neumann boundary conditions,
as in Theorem \ref{thm:asymL1Rbis} we have  $\Phi^{\theta} =0$,
$m_{u_{0}}=0$ and $\lim_{t \to \infty} S^{\theta}(t) u_{0} = 0$
  in  $L^{1}(\Omega)$. Then using Corollary \ref{cor:LpLq_estimates}
  and the semigroup property, we get
  \begin{displaymath}
   \big(\frac{t}{2}\big)^{\frac{N}{2}(\frac{1}{p}- 1)}  \norm{S^\theta(t)u_0}_{L^{p}(\Omega)}\leq
     C\norm{S^\theta(\frac{t}{2}) u_0}_{L^{1}(\Omega)} \to 0
  \end{displaymath}
as $t\to \infty$, which proves the result. 
      \end{proof}

\section*{Acknowledgments}
We would like to express our sincere gratitude to Fernando Quirós
(UAM, Spain) for
our conversations on the topic of exterior domains and asymptotic
behaviour. His guidance and discussions have greatly contributed to
the development of this work. We are also thankful to José A. Cañizo
and Institute of Mathematics of the University of Granada (IMAG, Spain) for
their warm hospitality during the visit of the first author there. The
engaging discussions on the aforementioned topic as well as on their
work in progress have been helpful for a wider understanding of the
topic covered on this document. In addition, special thanks go to Raúl
Ferreira (UCM, Spain) for his comments on the relationship between asymptotic
profiles and harmonic measure as well as to Pablo Hidalgo-Palencia
(ICMAT, Spain) for
the subsequent discussions and clarifications regarding this matter. I
appreciate the support and collaboration of these individuals, which
have significantly enriched the content and quality of this work.

\appendix

\begin{appendices}
  \section{Comparison Principles}
\label{app:comp}

  Now we present some monotonicity results for the solutions of the
parabolic problems above with
respect to the function $\theta$ and, in the case of Dirichlet
boundary conditions, with respect to the domain $\Omega$. 
But first, we need to define, given some $\theta-$boundary conditions
as in \eqref{eqn:thetabc}, the Dirichlet, and Robin/Neumann part of $\partial \Omega$.
We define the Dirichlet part of $\partial \Omega$ as
\begin{displaymath}
	\partial^D \Omega\defeq \{x \in \partial \Omega \ : \ \theta(x)=0\},
\end{displaymath}
the Robin/Neumann part of $\partial \Omega$ as
\begin{displaymath}
	\partial^R \Omega\defeq \{x \in \partial \Omega \ : \ 0<\theta(x)\leq1\}.
\end{displaymath}	
The conditions imposed on $\theta$ imply that $\partial^D \Omega$ is a
union of connected components of $\partial \Omega$.

Now we present some monotonicity results. The proof of the following
theorem can be found in \cite{DdTRB23}. 

\begin{theorem}
	\label{thm:neugeqdir}
	Let $\Omega\subset\RN$ be a domain with compact boundary and
        let ${u_1}_0, {u_2}_0\in L^{2}(\Omega)$,
        $f_{1}, f_{2} \in L^{1}((0,T), L^{2}(\Omega))$ and
        $g_{1}, g_{2} \in L^{1}((0,T), L^{2}(\partial\Omega))$ with
        $T>0$.  Finally, assume
        $u_1,u_2\in C^1((0,T), H^1_\theta(\Omega))\cap C([0,T],
        L^2(\Omega))$ are such that they are weak solutions of the
        problems
	\begin{displaymath}
          \label{eqn:heatgen2}
          \left\{
            \begin{aligned}
              \frac{\partial}{\partial t}u_i-\Lap u_{i} = f_i \quad & in \ \Omega\times(0,T) \\
              B_\theta(u_i)=g_i \quad & on \
              \partial\Omega\times (0,T) \\
              u_{i}={u_{i,0}} := u_{i}(0) \quad & in \
              \Omega\times\{0\} ,
            \end{aligned}
          \right. 
	\end{displaymath}
	for $i=1,2$, in the sense that $u_{i}= g_{i}$ on
        $\partial^{D}\Omega \times (0,T)$ and for any
        $\varphi\in C([0,T], H^1_\theta(\Omega))$,
	\begin{displaymath}
          \int_\Omega (u_{i})_t\varphi  + \int_\Omega \nabla u_{i} \nabla \varphi +
          \int_{\partial^R \Omega}   \cot(\frac{\pi}{2}\theta) u_{i} \varphi  =
          \int_{
            \partial^R\Omega}  \frac{g_{i}}{\sin(\frac{\pi}{2}\theta)}   \varphi + \int_\Omega f_{i} \varphi  \qquad t
          \in (0,T). 
	\end{displaymath}
	
	Then, if $f_1\geq f_2$, $g_1\geq g_2$ and ${u_{1,0}}\geq
        {u_{2,0}}$, we have
	\begin{displaymath}
          u_1\geq u_2  \qquad x\in \Omega, \ t \in (0,T) . 
	\end{displaymath}
      \end{theorem}

For  Dirichlet boundary conditions ($\theta\equiv 0$) we can also state a monotonicity result with respect to the domain.

\begin{theorem}
	\label{thm:compdom}
	Let $\Omega_1\subset\Omega_2\subset\RN$ domains and $0\leq
	u_i\in L^p(\Omega_i)$ for $i=1,2$ with $1\leq p \leq \infty$
	such that $0\leq u_1\leq \restr{u_2}{\Omega_{1}}$. Then, if we
	denote $S_{\Omega_{i}} ^0(t)$ the heat semigroup with zero Dirichlet
	boundary conditions in $\Omega_{i}$, we have that: 
	\begin{equation}
		S^{0}_{\Omega_1}(t)u_1\leq S^{0}_{\Omega_2}(t)u_2
		\qquad in \ \Omega_1, \ t>0.
	\end{equation}
	Therefore, the  heat kernels satisfy 
	\begin{equation}
		k^{0}_{\Omega_1}(x,y,t)\leq k^{0}_{\Omega_2}(x,y,t), \qquad x,y\in \Omega_1, \quad t>0.
	\end{equation}

	In particular, for any exterior domain with Dirichlet boundary
	conditions, we have     the Gaussian bound 
	\begin{equation}
		\label{eq:comparison_kernels_Dirichlet_RN}
		0 < k^0_{\Omega}(x,y,t) \leq k_{\R^{N}}(x,y,t) =
		\frac{e^{-\frac{\abs{x-y}^2}{4\pi t}}}{(4\pi t)^{N/2}} \qquad
		x,y\in \Omega , \quad t>0 . 
	\end{equation}
\end{theorem}
\begin{proof}
	Assume the initial data is smooth. Then,
	$u_{1}(t)=S^{0}_{\Omega_1}(t)u_1$ and  $u_{2}(t)=
	S^{0}_{\Omega_2}(t)u_2$ satisfy the heat  heat
	equation in $\Omega_1$, $\restr{u_{2}(t)}{\partial \Omega_{1}} \geq
	\restr{u_{1}(t)}{\partial \Omega_{1}} =0$ and the initial data
	satisfy $\restr{u_2}{\partial \Omega_1}\geq u_1$. Hence,
	using Theorem \ref{thm:neugeqdir} we obtain $S_{\Omega_1}(t)u_1\leq
	S_{\Omega_2}(t)u_2$. 
	
	Therefore for every $\varphi\in C^\infty_c(\Omega_1)$, we have that, for $x\in \Omega_1$
	\begin{equation}
		\int_{\Omega_2}k^{0}_{\Omega_2}(x,y,t)\varphi(y)dy=S^{0}_{\Omega_2}(t)\varphi(x)\geq
		S^{0}_{\Omega_1}(t)\varphi(x)=\int_{\Omega_1}k^{0}_{\Omega_1}(x,y,t)\varphi(y)dy, 
	\end{equation}
	so then $k^{0}_{\Omega_2}(x,y,t) \geq k^{0}_{\Omega_1}(x,y,t)$ for
	every $x,y\in \Omega_1$ and $t>0$. This immediately implies the result
	for non-smooth initial data by \eqref{eqn:ackrnpre}. 
\end{proof}

      The following theorem is helpful to compare solutions of
      parabolic equations in time-dependent domains. It allows
      $\theta-$boundary conditions and Dirichlet conditions
      in disjoint boundaries. The proof of this result in a more
      general setting can be found in \cite{friedman2008partial}
      Chapter 2, Theorems 1, 16 and 17.
      \begin{theorem}[Comparison Principle for Variable Domains]
	\label{thm:compvarneumann}
	Let $\Omega_{[t_0,t_1]}\subset \mathbb{R}^{N+1}$ be a space-time
	domain. Denote $\Omega_s=\{(x,t)\in \Omega_{[t_0,t_1]} :
	t=s\}$. Assume the boundary of $\Omega_{[t_0,t_1]}$ consists on
	the closure of a $N-dimensional$ domain $\Omega_{t_0}$ lying on
	$t=t_0$, the closure of a $N-dimensional$ domain $\Omega_{t_1}$
	lying on $t=t_1$ and a (not necessarily connected)
	$N-dimensional$ manifold $S$ lying on the strip $t_0\leq t\leq
	t_1$. Assume also that there is a curve $\gamma$ in
	$\Omega_{[t_0,t_1]}$ which connects $\Omega_{t_0}$ and $\Omega_{t_1}$ and
	whose $t$ coordinate is nondecreasing. Consider that
        $S=S_1\cup S_2$ where $S_1$ and $S_2$ are distinct connected
        manifolds. Assume $S_2$ is independent of time, that is
        $S_2=\Gamma\times (t_0,t_1)$. Furthermore, consider some
        $\theta$-boundary conditions (as in \eqref{eqn:thetabc}) on
        $\Gamma$. Then, for any $u_1, u_2\in C^{2,1}(\Omega_{[t_0,t_1]})$ such that
	\begin{equation}
          \left\{
            \begin{aligned}
              & \frac{\partial u_1}{\partial t}-\Lap u_1 > 0 \ \ \ && in \ \Omega_{[t_0,t_1]}, \\
              & \frac{\partial u_2}{\partial t}-\Lap u_2 \leq 0 \ \ \ && in \ \Omega_{[t_0,t_1]}, \\
              & u_1> u_2 \ \ \ && in \ \Omega_{t_0}\cup S_1, \\
              & B_\theta(u_1(t))> B_\theta(u_2(t)) \ \ \ && on \
              \Gamma, \ \forall t\in (t_0,t_1),
            \end{aligned}
          \right.
	\end{equation}
	we have that
	\begin{equation}
          u_1> u_2 \ \ \ in \ \Omega_{[t_0,t_1]}.
	\end{equation}
      \end{theorem}
      \section{Removable Singularities}
      Here we present a result which allows us to remove singularities
      in some cases when we have solutions of the heat equation in the
      whole space except in a point. This topic is further studied in
      \cite{aronsonremovable}.
      \begin{theorem}
	\label{thm:removablesing}
	Let $N\geq 2$ and
        $u\in L^\infty([t_0,t_1],L^\infty(\RN\backslash\{0\}))$ a
        bounded solution of the heat equation in
        $\RN\backslash\{0\}\times [t_0,t_1]$, that is,
        $u\in C^{(2,1)}_{x,t}(\RN\backslash\{0\}\times [t_0,t_1])$
        such that
	\begin{equation}
          u_t(x,t)-\Lap u(x,t) = 0 \qquad \forall x\in \RN\backslash\{0\}, \quad \forall t\in [t_0,t_1],
	\end{equation}
        then, $u$ can be extended so that
        $u\in C^{(2,1)}(\RN\times [t_0,t_1])$ and it is a solution of
        the heat equation in the whole space
	\begin{equation}
          u_t(x,t)-\Lap u(x,t) = 0 \qquad \forall x\in \RN, \quad \forall t\in [t_0,t_1].
	\end{equation}
      \end{theorem}
      \begin{proof}
        It is just enough to apply Theorem 1 from
        \cite{aronsonremovable} with the set $K=\{0\}\times [t_0,t_1]$
        which is a $(2,\infty)$-null set just because a point has zero
        capacity in $\RN$ for $N\geq 2$.

	This is easily seen considering the functions:
        \begin{equation}
          \psi_\alpha(x) = \left\{ \begin{aligned}
              & 1-\abs{x}^{\alpha} \qquad && \forall \abs{x}\leq 1 \\
              & 0 \qquad \qquad \quad \  && \forall \abs{x}\geq 1  
            \end{aligned}
          \right. ,
        \end{equation}
        because $\psi_n(0)=1$, $\psi_n\in H^1_0(\RN)$ (for $N\geq 2$)
        and
        \begin{equation}
          \lim_{\alpha\to 0} \int_{\RN} \abs{\nabla \psi_{\alpha}}^2 = \lim_{\alpha\to 0} \int_{B(0,1)} \alpha^2 \abs{x}^{2\alpha-2} = \lim_{\alpha\to 0} \dfrac{\alpha^2}{2\alpha+n-2}\omega_{N-1}=0,
        \end{equation}
        where $\omega_{N-1}$ is the volume of the $(N-1)-dimensional$
        unit ball.
      \end{proof}

      \section{Schauder Estimates}
      Here we present some parabolic Schauder estimates, which allow
      us to estimate the derivatives of a solution of the heat
      equation just with the $L^\infty$ norm of the solutions. These
      are classical results which can be found, for example, in
      \cite{friedman2008partial} Chapter 3 Theorem 5.

      \begin{theorem}
	\label{thm:schauderest}
	Let $K\subset \RN$ a domain, $Q\defeq K\times [T_1,T_2]$ and
        $v\in L^\infty(Q)\cap C^{\infty}(Q)$ be a solution of the heat equation. Define,
        for any $(x,t)\in Q$ the parabolic distance
        $d_{(x,t)}=\inf\{(\abs{x-\bar{x}}^2+\abs{t-\bar{t}})^{1/2} \ :
        \ (\bar{x},\bar{t})\in \partial Q \backslash\{(x,T_2):x\in
        \Omega\} \}$. Then,
	\begin{equation}
          d_{(x,t)}\abs{Dv(x,t)}+d^2_{(x,t)}\abs{D^2v(x,t)}\leq C\norm{v}_{L^\infty(Q)} \qquad\forall (x,t)\in Q,
	\end{equation}
	where $C$ is independent of $v$, $x$ $t$, $K$, $T_1$ and
        $T_2$, $Dv$ represent any first order spatial derivative of
        $v$ and $D^2v$ any second order spatial derivative of $v$.
      \end{theorem}

      In the same way, in the elliptic framework, we have also
      Schauder estimates.  The proof of the following results can be
      found for example in \cite{gilbarg2015elliptic} Theorem 4.6:
      \begin{theorem}
	\label{thm:4.6}
	Let $\Omega\subset\RN$ and $u\in C^2(\Omega)$ such that
        $$\Lap u(x) =0 \ \ \ \forall x\in\Omega.$$ Then, for
        $x_0\in \Omega$ and any two concentric balls
        $B_1\defeq B(x_0,R)$ and
        $B_{2}\defeq B(x_0, 2R)\subset \subset \Omega$, we have
	\begin{equation}
          \label{eqn:thm4.6}
          R\abs{Du}_{B_1}+R^2\abs{D^2u}_{B_1}\leq C\norm{u}_{L^\infty(B_2)},
	\end{equation}
        where we denote
        $\abs{Du}_{B_1}=\max_{i}\norm{D_i u}_{L^\infty(B_1)}$,
        $\abs{D^2u}_{B_1}=\max_{i,j}\norm{D_{ij} u}_{L^\infty(B_1)}$
        and $C$ is a constant independent on $u$, $x_0$ and $R$.
      \end{theorem}

    \end{appendices}

\bibliographystyle{alpha-mod}
\newcommand{\etalchar}[1]{$^{#1}$}

\end{document}